\documentclass{elsart}
\usepackage{array,amsmath,amstext,amssymb,amsbsy,graphicx}
\usepackage{pdfsync}
 \usepackage{color}
\usepackage{txfonts}
\usepackage{subfigure}

\newtheorem{definition}{Definition}
\newtheorem{lemma}{Lemma}
\newtheorem{proposition}{Proposition}
\newtheorem{theorem}{Theorem}
\newtheorem{remark}{Remark}
\newtheorem{property}{Property}
\newtheorem{example}{Example}
\newenvironment{proof}{\noindent \newline { Proof.}}
{\hfill \mbox{$\square$ } \newline}
\begin{document}
\begin{frontmatter}
\title{Well-posedness and numerical algorithm for the tempered fractional ordinary differential equations}
\author[CL]{Can Li} 
\ead{mathlican@xaut.edu.cn}
\address[CL]{Beijing Computational Science Research Center, Beijing 10084, P.R. China.
\\Department of Applied Mathematics, School of Sciences,
Xi'an University of Technology, Xi'an, Shaanxi 710054, P.R. China.
}
\author[WHD]{Weihua Deng}
\address[WHD]{ School of Mathematics and Statistics,\\
Gansu Key Laboratory of Applied Mathematics and Complex Systems,
Lanzhou University, Lanzhou 730000, P.R. China.}
\ead{dengwh@lzu.edu.cn}

\author[ZLJ]{Lijing Zhao}
\address[ZLJ]{ School of Mathematics and Statistics,\\
Gansu Key Laboratory of Applied Mathematics and Complex Systems,
Lanzhou University, Lanzhou 730000, P.R. China.}
\ead{zhaolj10@lzu.edu.cn}
\date{}

\maketitle
\begin{abstract}
Trapped dynamics widely appears in nature, e.g., the motion of particles in viscous cytoplasm. The famous continuous time random walk (CTRW) model with power law waiting time distribution ({\em having diverging first moment}) describes this phenomenon. Because of the finite lifetime of biological particles, sometimes it is necessary to temper the power law measure such that the waiting time measure has convergent first moment. Then the time operator of the Fokker-Planck equation corresponding to the CTRW model with tempered waiting time measure is the so-called tempered fractional derivative. This paper focus on discussing the properties of the time tempered fractional derivative, and studying the well-posedness and the Jacobi-predictor-corrector algorithm for the tempered fractional ordinary differential equation. By adjusting the parameter of the proposed algorithm, any desired convergence order can be obtained and the computational cost linearly increases with time. And the effectiveness of the algorithm is numerically confirmed.
\end{abstract}
\end{frontmatter}

\section{Introduction}
The fractional calculus  has a long history.
The origin of fractional calculus can be traced
back to the letter between Leibniz and L'H\^{o}pital in 1695.
In the past three centuries, the development of the theories of
fractional calculus is well contributed
by many mathematicians and physicists. And from the last century, the books covering fractional calculus
began to emerge, such as Oldham and Spanier (1974),
Samko, Kilbas and Marichev (1993),
Podlubny (1999), and so on.
In recent years, more theories and experiments show that a broad
range of non-classical phenomena appeared in the applied sciences and
engineering can be described by fractional calculus
\cite{Podlubny:99,Metzler:00,Samko:93}.
Because of its good mathematical features,
nowadays fractional calculus has become a powerful tool in depicting
the anomalous kinetics which arises in physics, chemistry,
biology, finance, and other complex dynamics \cite{Metzler:00}.
In practical  applications, several different kinds of fractional derivatives,
such as  Riemann-Liouville  fractional derivative,
Caputo fractional derivative \cite{Podlubny:99,Samko:93},
Riesz fractional derivative\cite{Samko:93}, and
Hilfer fractional derivative\cite{Hilfer:00,Tomovski:12}
are introduced.

One of the typical applications for fractional
calculus is the description of anomalous diffusion behavior of living particles; and the tempered fractional calculus describes the transition between normal and anomalous diffusions (or the anomalous diffusion in finite time or bounded physical space).
In the continuous time random walk (CTRW) model,
for a L\'{e}vy flight particle, the scaling limit of the CTRW
with a jump distribution function $\phi(x)\sim x^{-(1+\alpha)} (1 < \alpha<2)$
exhibits superdiffusive dynamics.
The corresponding stable L\'{e}vy distribution for particle displacement contains arbitrarily
large jumps and has divergent spatial moments. However, the infinite spatial moments
may not be feasible for some physical processes \cite{Cartea:07}. One way
to overcome the divergence of the moments of L\'{e}vy  distributions
in transport models is to exponentially temper the L\'{e}vy measure.
Then the space fractional operator will be replaced by the spatially
tempered fractional operator in the corresponding models \cite{Cartea:07a,Cartea:07, Sabzikar:15}. This paper concentrates on the time tempered fractional derivative, which arises in the Fokker-Planck equation corresponding to the CTRW model with tempered power law waiting time distribution \cite{Schmidt:10,Gajda:11}. Tempering the power law waiting time measure makes its first moment finite and the trapped dynamics more physical. Sometimes it is necessary/reasonable to make the first moment of the waiting time measure finite, e.g., the biological particles moving in viscous cytoplasm and displaying trapped dynamical behavior just have finite lifetime.
%
The time tempered diffusion dynamics describes the coexistence/transition of subdiffusion and normal diffusion phenomenon (or the subdiffusion in finite time) which was empirically confirmed in a number of systems \cite{Cartea:07,Meerschaert:09}.
More applications for the tempered fractional derivatives and tempered differential equations can be found, for instance, in  poroelasticity \cite{Hanyga:2001},
finance \cite{Cartea:07a}, ground water hydrology \cite{Meerschaert:09,Meerschaert:12},
and geophysical flows \cite{Meerschaert:14}.


Tempered fractional calculus can be recognized as the generalization of fractional calculus.
To the best of our knowledge,  the definitions of fractional integration with weak
singular and exponential kernels were firstly reported in Buschman's earlier work \cite{Buschman:72}. For the other different definitions of the tempered fractional integration, see the books \cite{Srivastava:77,Samko:93,Meerschaert:12} and references therein.
This work continues previous efforts \cite{Li:14} to explore the time tempered fractional derivative. The well-posedness, including existence, uniqueness, and stability, of the tempered fractional ordinary differential equation (ODE) is discussed, and the properties of the time tempered fractional derivative are analyzed. Then the Jacobi-predictor-corrector algorithm for the tempered fractional ODE is provided, which has the striking benefits: 1. any desired convergence order can be obtained by simply adjusting the parameter (the number of interpolation points); the computational cost increases linearly with the time $t$ instead of $t^2$ usually taken place for nonlocal time dependent problem. And extensive numerical experiments are performed to confirm these advantages.


In Section \ref{secdptfc}, we introduce the definitions and show the properties of the tempered fractional calculus, including the generalizations of the tempered fractional derivatives in the Riemann-Liouville and Caputo sense, and the composite property. More basic properties are listed and proved in Appendix A; the expression and properties of the tempered fractional calculus in Laplace space are proposed and proved in Appendix B. In Section \ref{secfour}, we discuss the initial value problem of the tempered fractional ODE: first derive the Volterra integral formulation of the tempered fractional ODE; then prove the well-posedness of the considered problem. The Jacobi-Predictor-Corrector algorithm for the tempered fractional ODE is designed and discussed in Section \ref{sec:algorithm}, and two numerical examples are solved by the algorithm to show its powerfulness.

\section{Preliminaries}\label{secdptfc}
In this section, we first give the definitions and some properties of the tempered fractional calculus.
Let $[a,b]$ be a finite interval on the real line $\mathbb{R}$. Denote $L([a,b])$ as the integrable space which includes the Lebesgue measurable functions on the finite interval $[a,b]$, i.e.,
$$L([a,b])=\bigg\{u:\|u\|_{L([a,b])}=\int_{a}^{b}|u(t)|dt<\infty\bigg\}.$$
And let $AC[a,b]$ be the space of real-values functions $u(t)$ which are absolutely continuous on $[a,b]$. For $n\in \mathbb{N}^{+}$, we denote $AC^n[a,b]$ as the space of real-values functions $u(t)$ which have continuous derivatives up
to order $n-1$ on $[a,b]$ such that $\frac{d^{n-1}u(t)}{dx^{n-1}}\in AC[a,b]$, i.e.,
$$AC^n[a,b]=\bigg\{u:[a,b]\rightarrow\mathbb{R},\frac{d^{n-1}}{dx^{n-1}}u(t)\in AC[a,b]\bigg\}.$$ And denote by $C^n[a,b]$ the  space of functions $u(t)$ which are $n$ times continuously differentiable on $[a,b]$.

\begin{definition} [\rm Riemann-Liouville tempered fractional integral \cite{Buschman:72,Cartea:07}] \label{Def1}
Suppose that the real function $u(t)$ is piecewise continuous on $(a,b)$ and  $u(t)\in L([a,b])$,  $\sigma>0, \lambda\ge 0$. The  Riemann-Liouville tempered fractional integral of order $\sigma$ is defined to be
\begin{equation}\label{TRLI}
_{a}I_t^{\sigma,\lambda}u(t)=e^{-\lambda t} {_{a}I}_t^{\sigma}\big(e^{\lambda t}u(t)\big)=\frac{1}{\Gamma(\sigma)}\int_{a}^te^{-\lambda( t-s)}(t-s)^{\sigma-1}u(s)ds,
\end{equation}
where $_{a}I_t^{\sigma}$ denotes the Riemann-Liouville fractional integral
\begin{equation}\label{RLIx}
_{a}I_t^{\sigma}u(t)=\frac{1}{\Gamma(\sigma)}\int_{a}^t(t-s)^{\sigma-1}u(s)ds.
\end{equation}
\end{definition}
Obviously, the tempered fractional integral \eqref{TRLI} reduces to the Riemann-Liouville fractional integral if $\lambda=0$. In practical applications, sometimes the fractional integral \eqref{TRLI} is represented as $_{a}D_t^{-\sigma,\lambda}u(t)$.

\begin{definition}[\rm Riemann-Liouville tempered  fractional derivative \cite{Baeumera:10,Cartea:07}]\label{defrlt}
For $n-1<\alpha<n,n\in \mathbb{N}^{+}, \lambda\ge 0$. 
The Riemann-Liouville tempered  fractional derivative
is defined by
\begin{equation}\label{TRLD}
_{a}D_t^{\alpha,\lambda}u(t)
=e^{-\lambda t}{{_{a}}D_t^{\alpha}}\big( e^{\lambda t}u(t)\big)=\frac{e^{-\lambda t}}{\Gamma(n-\alpha)}\frac{d^n}{d
t^n}\int_{a}^t\frac{e^{\lambda s}u(s)}{(t-s)^{\alpha-n+1}}ds,
\end{equation}
where $_{a}D_t^{\alpha}(e^{\lambda t}u(t))$ denotes the Riemann-Liouville fractional derivative \cite{Podlubny:99}
\begin{equation}\label{RLD}
      _{a}D_t^{\alpha}(e^{\lambda t}u(t))=\frac{d^n}{d
          t^n}\big(_{a}I_t^{n-\alpha}(e^{\lambda t}u(t))\big)=\frac{1}{\Gamma(n-\alpha)}\frac{d^n}{d
          t^n}\int_{a}^t\frac{(e^{\lambda s}u(s))}{(t-s)^{\alpha-n+1}}ds.
    \end{equation}
\end{definition}

\begin{remark}[\cite{Baeumera:10}]\label{def44}
The variants of the  Riemann-Liouville tempered
fractional derivatives are defined as
\begin{equation}\label{Lrl}
_{a}\textbf{D}_t^{\alpha,\lambda}u(t)=
\begin{cases}
\displaystyle
_{a}D_t^{\alpha,\lambda}u(t)-\lambda^{\alpha}u(t),& \text{$0<\alpha<1$},\\
\displaystyle
_{a}D_t^{\alpha,\lambda}u(t)-\alpha\lambda^{\alpha-1}\frac{du(t)}{dt}-\lambda^{\alpha}u(t),& \text{$1<\alpha<2$}.
\end{cases}
\end{equation}
\end{remark}
\begin{definition}[\rm fractional substantial derivative \cite{Friedrich:06,Turgeman:09,Chen:13}]\label{defrlsf}
For $n-1<\alpha<n,n\in \mathbb{N}^{+}$, 
and $\lambda(x)$ being any given function defined in space domain.
The Riemann-Liouville fractional substantial derivative is defined by
\begin{equation}\label{suTRLD}
D_s^{\alpha,\lambda(x)}u(t)
=\bigg(\frac{d}{d
t}+\lambda(x)\bigg)^n{_{a}I_t^{n-\alpha,\lambda(x)}}u(t)=\bigg(\frac{d}{d
t}+\lambda(x)\bigg)^n\int_{a}^t\frac{e^{-\lambda(x) \cdot (t-s)}u(s)}{(t-s)^{\alpha-n+1}}ds,
\end{equation}
where ${_{a}I_t^{n-\alpha,\lambda(x)}}$ denotes the Riemann-Liouville fractional integral and
\begin{equation}\label{ddlammda}
\bigg(\frac{d}{d t}+\lambda(x)\bigg)^n=\underbrace{\bigg(\frac{d}{d t}+\lambda(x)\bigg)\cdots\bigg(\frac{d}{d t}+\lambda(x)\bigg)}_{n\ times}.
\end{equation}
\end{definition}

\begin{remark}\label{sgltem}
The fractional substantial derivative \eqref{suTRLD} is equivalent
to the Riemann-Liouville tempered  fractional derivative \eqref{TRLD} if $\lambda(x)$ is a nonnegative constant function. In fact, using integration by parts leads to
\begin{equation*}
\begin{split}
    & \bigg(\frac{d}{d
t}+\lambda(x)\bigg)^{n}\bigg[\int_{a}^t\frac{e^{-\lambda(x)\cdot (t-s)}u(s)}{(t-s)^{\alpha-n+1}}ds\bigg]\\
     &=\bigg(\frac{d}{d
t}+\lambda(x)\bigg)^{n-1}\bigg[\bigg(\frac{d}{d
t}+\lambda(x)\bigg)\int_{a}^t\frac{e^{-\lambda(x) \cdot (t-s)}u(s)}{(t-s)^{\alpha-n+1}}ds\bigg]\\
     &=\bigg(\frac{d}{d
t}+\lambda(x)\bigg)^{n-1}\bigg[e^{-\lambda(x) t}\frac{d}{d
t}\int_{a}^t\frac{e^{\lambda(x) s}u(s)}{(t-s)^{\alpha-n+1}}ds\bigg]\\
&=\bigg(\frac{d}{d
t}+\lambda(x)\bigg)^{n-2}\bigg[e^{-\lambda(x) t}\frac{d^2}{d
t^2}\int_{a}^t\frac{e^{\lambda(x) s}u(s)}{(t-s)^{\alpha-n+1}}ds\bigg]\\
     &=\cdots\\
     &={_{a}D_t^{\alpha,\lambda(x)}}u(t).
\end{split}
\end{equation*}
The tempered $n$-th order derivative of $u(t)$ equals to $\left(\frac{d}{d
t}+\lambda\right)^{n}u(t)$, which can be simply/resonably denoted as ${D^{n,\lambda}}u(t)$.
\end{remark}

\begin{definition}[\rm Caputo tempered  fractional derivative \cite{Samko:93,Tatar:04}]\label{def33c}
For $n-1<\alpha<n,n\in \mathbb{N}^{+},\lambda\ge0$. 
The Caputo tempered fractional derivative is defined as
    \begin{equation}\label{TCD}
      {_{a}^{C}D}_t^{\alpha,\lambda}u(t)=e^{-\lambda t}~{_{a}^{C}D}_t^{\alpha}\big(e^{\lambda t}u(t)\big)=\frac{e^{-\lambda t}}{\Gamma(n-\alpha)}\int_{a}^t\frac{1}{(t-s)^{\alpha-n+1}}\frac{d^n (e^{\lambda s}u(s))}{d
          s^n}ds,
    \end{equation}
where
 ${_{a}^{C}D}_t^{\alpha,\lambda}(e^{\lambda t}u(t))$  denotes
 the Caputo fractional derivative \cite{Podlubny:99}
\begin{equation}\label{CD}
      {_{a}^{C}D}_t^{\alpha}(e^{\lambda t}u(t))=\frac{1}{\Gamma(n-\alpha)}\int_{a}^t\frac{1}{(t-s)^{\alpha-n+1}}\frac{d^n(e^{\lambda s}u(s))}{d
          s^n}ds.
    \end{equation}
\end{definition}
\begin{remark}
The equivalent forms of Riemann-Liouville tempered fractional derivative (\ref{TRLD}) and Caputo tempered fractional derivative (\ref{TCD}) are ${_{a}D}_t^{\alpha,\lambda}u(t)={D}^{n,\lambda}{_{a}I_t^{n-\alpha,\lambda}}u(t)$  and ${_{a}^{C}D}_t^{\alpha,\lambda}u(t)={_{a}I_t^{n-\alpha,\lambda}}\, {D}^{n,\lambda}u(t)$, respectively.
\end{remark}

Note that when $\lambda=0$, the Riemann-Liouville (Caputo) tempered fractional derivative
reduces to the Riemann-Liouville (Caputo) fractional derivative.
\begin{proposition} \label {lemma2as}
Let $u(t)\in AC^{n}[a,b]$ and $n-1<\alpha<n$. Then for all $t \geq a$,
holds
\begin{equation}\label{glcrelation}
{_{a}^{C}D_t^{\alpha,\lambda}} \big(u(t)\big)={_{a}D_t^{\alpha,\lambda}}\big(u(t)\big)-\sum_{k=0}^{n-1}\frac{e^{-\lambda t}(t-a)^{k-\alpha}}{\Gamma(k-\alpha+1)}\bigg[\frac{d^{k}}{dt^{k}}\big(e^{\lambda t}u(t)\big)\bigg|_{t=a}\bigg].
\end{equation}
\end{proposition}
\begin{proof}
Take $v(t)=e^{\lambda t}u(t)$ in the equation for the Riemann-Liouville and Caputo fractional derivatives \cite{Podlubny:99,Li:07,Samko:93}
$$_{a}^{C}D_t^{\alpha} v(t)={_{a}D_t^{\alpha}}\bigg(v(t)-\sum_{k=0}^{n-1}\frac{(t-a)^k}{k!}\frac{d^{k}v(t)}{dt^{k}}\big|_{t=a}\bigg),$$
yielding
$$_{a}^{C}D_t^{\alpha} \big(e^{\lambda t}u(t)\big)={_{a}D_t^{\alpha}}\bigg(e^{\lambda t}u(t)-\sum_{k=0}^{n-1}\frac{(t-a)^k}{k!}\frac{d^{k}}{dt^{k}}\big(e^{\lambda t}u(t)\big)\big|_{t=a}\bigg).$$
Multiplying both sides of the above equation by $e^{-\lambda t}$, we obtain
$$e^{-\lambda t}~{_{a}^{C}D_t^{\alpha}} \big(e^{\lambda t}u(t)\big)=e^{-\lambda t}{_{a}D_t^{\alpha}}\bigg(e^{\lambda t}u(t)-\sum_{k=0}^{n-1}\frac{(t-a)^k}{k!}\frac{d^{k}}{dt^{k}}\big(e^{\lambda t}u(t)\big)\big|_{t=a}\bigg).$$
Furthermore, using the definitions of Riemann-Liouville and Caputo tempered fractional derivatives, we get that
\begin{equation}\label{glcrelationa}
{_{a}^{C}D_t^{\alpha,\lambda}} \big(u(t)\big)={_{a}D_t^{\alpha,\lambda}}\big(u(t)\big)-\sum_{k=0}^{n-1}e^{-\lambda t}{_{a}D_t^{\alpha}}\bigg(\frac{(t-a)^k}{k!}\bigg)\bigg[\frac{d^{k}}{dt^{k}}\big(e^{\lambda t}u(t)\big)\bigg|_{t=a}\bigg].
\end{equation}
Using the linearity properties presented in Proposition \ref{lemma2ab} and the formula of power function
$${_{a}D_t^{\alpha}}\bigg(\frac{(t-a)^k}{k!}\bigg)=\frac{1}{k!}{_{a}D_t^{\alpha}}\bigg((t-a)^k\bigg)
=\frac{1}{k!}\frac{\Gamma(k+1)(t-a)^{k-\alpha}}{\Gamma(k-\alpha+1)},$$
we  deduce  the desired result from \eqref{glcrelationa}.
\end{proof}

\begin{proposition} \label {lemma2a}(Composite properties)

(1) Let $u(x) \in L([a,b])$ and $I^{n-\alpha,\lambda} u(t) \in AC^n[a,b]$. Then the Riemann-Liouville tempered fractional derivative and integral have the composite properties
\begin{equation}\label{rca}
_{a}I_t^{\alpha,\lambda}[_{a}D_t^{\alpha,\lambda} u(t)]=u(t)-\sum_{k=0}^{n-1}\frac{e^{-\lambda t}(t-a)^{\alpha-k-1}}{\Gamma(\alpha-k)}\big[{_{a}D}_t^{\alpha-k-1}(e^{\lambda t}u(t))\big|_{t=a}\big],
\end{equation}
and
\begin{equation}\label{rcb}
_{a}D_t^{\alpha,\lambda}[_{a}I_t^{\alpha,\lambda} u(t)]=u(t).
\end{equation}
(2) Let $u(t)\in AC^n[a,b]$ and $n-1<\alpha<n$. Then the Caputo tempered fractional derivative and the Riemann-Liouville tempered fractional integral have the composite properties
\begin{equation}\label{rcaa}
_{a}I_t^{\alpha,\lambda}[_{a}^{C}D_t^{\alpha,\lambda} u(t)]=u(t)-\sum_{k=0}^{n-1}e^{-\lambda t}\frac{(t-a)^k}{k!}\bigg[\frac{d^k (e^{\lambda t}u(t))}{dt^k}\bigg|_{t=a}\bigg],
\end{equation}
and
\begin{equation}\label{rcab}
{_{a}^{C}D}_t^{\alpha,\lambda}[_{a}I_t^{\alpha,\lambda} u(t)]=u(t) ~~~{\rm if}~~ \alpha \in (0,1).
\end{equation}
\end{proposition}
\begin{proof}
From the definitions of Riemann-Liouville tempered fractional integral
and derivative, we have
\begin{equation}\label{imca}
\begin{split}
    _{a}I_t^{\alpha,\lambda} [_{a}D_t^{\alpha,\lambda} u(t)]&=
     e^{-\lambda t}{{_{a}}I_t^{\alpha}}\big[ e^{\lambda t}(_{a}D_t^{\alpha,\lambda} u(t))\big]      \\
     &=e^{-\lambda t}{{_{a}}I_t^{\alpha}}\big[ e^{\lambda t}\big(e^{-\lambda t}{{_{a}}D_t^{\alpha}}( e^{\lambda t}u(t))\big)\big]    \\
     &=e^{-\lambda t}
     \underbrace{{{_{a}}I_t^{\alpha}}\big[{{_{a}}D_t^{\alpha}}( e^{\lambda t}u(t))\big]}_{(I)}.
\end{split}
\end{equation}
Thanks to the composition formula \cite{Podlubny:99,Li:07,Samko:93}
\begin{equation*}
_{a}I_t^{\alpha}[_{a}D_t^{\alpha} v(t)]=v(t)-\sum_{k=0}^{n-1}\frac{(t-a)^{\alpha-k-1}}{\Gamma(\alpha-k)}\big[{_{a}D}_t^{\alpha-k-1}(v(t))\big|_{t=a}\big],
\end{equation*}
we get
\begin{equation*}
(I)={_{a}I_t^{\alpha}}[_{a}D_t^{\alpha} e^{\lambda t}u(t)]=e^{\lambda t}u(t)-\sum_{k=0}^{n-1}\frac{(t-a)^{\alpha-k-1}}{\Gamma(\alpha-k)}\big[{_{a}D}_t^{\alpha-k-1}(e^{\lambda t}u(t))\big|_{t=a}\big].
\end{equation*}
Inserting the above formula into \eqref{imca} leads to \eqref{rca}.

Again from the definitions of Riemann-Liouville tempered fractional integral
and derivative, there exists
\begin{equation*}
\begin{split}
     _{a}D_t^{\alpha,\lambda}[_{a}I_t^{\alpha,\lambda} u(t)]&=
     e^{-\lambda t}{{_{a}}D_t^{\alpha}}\big[ e^{\lambda t}(_{a}I_t^{\alpha,\lambda} u(t))\big]      \\
     &=e^{-\lambda t}{{_{a}}D_t^{\alpha}}\big[ e^{\lambda t}\big(e^{-\lambda t}{{_{a}}I_t^{\alpha}}( e^{\lambda t}u(t))\big)\big]    \\
     &=e^{-\lambda t}{{_{a}}D_t^{\alpha}}\big[{{_{a}}I_t^{\alpha}}( e^{\lambda t}u(t))\big].
\end{split}
\end{equation*}
Furthermore, using the composite properties of Riemann-Liouville fractional integral
and derivative \cite{Podlubny:99,Li:07,Samko:93}
\begin{equation}\label{igr}
{{_{a}}D_t^{\alpha}}\big[{{_{a}}I_t^{\alpha}}( v(t))\big]=v(t),
\end{equation} we get
\begin{equation*}
_{a}D_t^{\alpha,\lambda}[_{a}I_t^{\alpha,\lambda} u(t)]
=e^{-\lambda t}{{_{a}}D_t^{\alpha}}\big[{{_{a}}I_t^{\alpha}}( e^{\lambda t}u(t))\big]=u(t),
\end{equation*}
by taking $v(t)=e^{\lambda t}u(t)$ in Eq. \eqref{igr}.

Similarly, using the  composite properties of Caputo fractional derivative \cite{Podlubny:99,Li:07,Samko:93}
\begin{equation*}
_{a}I_t^{\alpha}[{_{a}^{C}D_t^{\alpha}} v(t)]=v(t)-\sum_{k=0}^{n-1}\frac{(t-a)^k}{k!}\bigg[\frac{d^kv(t)}{dt^k}\bigg]\bigg|_{t=a},
\end{equation*}
and
\begin{equation}\label{icr}
{{_{a}^{C}}D_t^{\alpha}}\big[{{_{a}}I_t^{\alpha}}( v(t))\big]=v(t),  ~~~{\rm if}~~ \alpha \in (0,1),
\end{equation}
we can get \eqref{rcaa} and \eqref{rcab}.

\end{proof}
\begin{remark}
For a constant $C$,
\begin{equation}\label{Tcfa}
{_{a}D_t^{\alpha,\lambda}}C=Ce^{-\lambda t}{_aD_t^\alpha e^{\lambda t}}, \quad {_{a}^CD_t^{\alpha,\lambda}}C=Ce^{-\lambda t}\,{_a^CD_t^\alpha e^{\lambda t}}.
\end{equation}
Obviously, $_{a}D_t^{\alpha,\lambda}(C)\neq {_{a}^{C}D}_t^{\alpha,\lambda}(C)$. And ${_{a}^{C}D}_t^{\alpha,\lambda}(C)$ is no longer equal to zero, being different from ${_{a}^{C}D}_t^{\alpha}(C)\equiv 0$.


\end{remark}

\section{Well-posedness of the tempered fractional ordinary differential equations}\label{secfour}
In this section, we consider the ODEs with Riemann-Liouville and Caputo tempered fractional derivatives, respectively, i.e.,
\begin{equation}\label{pRrla}
\begin{cases}
\displaystyle
_{a}D_t^{\alpha,\lambda}u(t)=f(t,u(t)),~ \text{$n-1<\alpha<n,\lambda\geq0,$}\\
\displaystyle
\big[{_{a}D}_t^{\alpha-k-1}\big(e^{\lambda t}u(t)\big)\big]\big|_{t=a}=g_k,~\text{$k=0,1,2,\cdots,n-1$},
\end{cases}
\end{equation}
and
\begin{equation}\label{pRrlb}
\begin{cases}
\displaystyle
 {_{a}^{C}D}_t^{\alpha,\lambda}u(t)=f(t,u(t)),~ \text{$n-1<\alpha<n,\lambda\geq0,$}\\
\displaystyle
\bigg[\frac{d^{k}}{dt^{k}}(e^{\lambda t}u(t))\bigg]\bigg|_{t=a}=c_k,~\text{$k=0,1,2,\cdots,n-1$}.
\end{cases}
\end{equation}
The Cauchy problems \eqref{pRrla} and \eqref{pRrlb} can be converted to the equivalent Volterra integral equations of the second kind under some proper conditions.
\begin{lemma}\label{gldf}
If the function $f(t,u(t))$ and $u(t)$ belong to $L([a,b])$, then $u(t)$ is solution of the initial value problem \eqref{pRrla}
if and only if $u(t)$ is the solution of the Volterra integral equation of the second kind
 \begin{equation}\label{integral}
u(t)=\sum_{k=0}^{n-1}g_k\frac{e^{-\lambda t}(t-a)^{\alpha-k-1}}{\Gamma(\alpha-k)}
+ \frac{1}{\Gamma(\alpha)}\int_{a}^t e^{-\lambda( t-s)} (t-s)^{\alpha-1} f(s,u(s))ds.
\end{equation}
In particular, if $0<\alpha<1$, then $u(t)$ satisfies the Cauchy problem \eqref{pRrla} if and only if
$u(t)$ satisfies the following integral equation
\begin{equation}\label{integralb}
u(t)=
g_0\frac{e^{-\lambda t}(t-a)^{\alpha-1}  }{\Gamma(\alpha)}
+
\frac{1}{\Gamma(\alpha)}\int_{a}^t e^{-\lambda( t-s)} (t-s)^{\alpha-1} f(s,u(s))ds.
\end{equation}
\end{lemma}
\begin{proof}
For the linear Cauchy problems \eqref{pRrla} and \eqref{pRrlb}, the conclusion is directly reached by
the Laplace transform given in Appendix B. Now we prove the more general case.

$Necessity.$
Performing the integral operator ${_{a}}I_t^{\alpha,\lambda}$ on both sides of the first equation of \eqref{pRrla}, we have
\begin{equation*}
u(t)=\sum_{k=0}^{n-1}g_k\frac{e^{-\lambda t}(t-a)^{\alpha-k-1}}{\Gamma(\alpha-k)}
+ \frac{1}{\Gamma(\alpha)}\int_{a}^t e^{-\lambda( t-s)} (t-s)^{\alpha-1} f(s,u(s))ds,
\end{equation*}
where we use the composite property (1) given in Proposition \ref{lemma2a}.
 Then Eq. \eqref{integral} is obtained.

$Sufficiency.$
Applying the operator $_{a}D_t^{\alpha,\lambda}$ to both sides of Eq. \eqref{integral} results in
\begin{equation}\label{integrala}
_{a}D_t^{\alpha,\lambda}u(t)=\sum_{k=0}^{n-1}g_k\frac{_{a}D_t^{\alpha,\lambda}(e^{-\lambda t}(t-a)^{\alpha-k-1} ) }{\Gamma(\alpha-k)}
+ {_{a}D_t^{\alpha,\lambda}}{_{a}I_t^{\alpha,\lambda}}f(t,u(t))=f(t,u(t)),
\end{equation}
where we use the fact
\begin{equation*}\label{power}
\frac{_{a}D_t^{\alpha,\lambda}(e^{-\lambda t}(t-a)^{\alpha-k-1} ) }{\Gamma(\alpha-k)}
=\frac{e^{-\lambda t}(t-a)^{\alpha-k-1-\alpha}}{\Gamma(-k)}=\frac{e^{-\lambda t}(t-a)^{-k-1}}{\infty}=0,\,k=0,1,2,\cdots,n-1,
  \end{equation*}
and the composite property \eqref{rcb}.
Now we show that the solution of \eqref{integral} satisfies the initial condition given in  Eq. \eqref{pRrla}.
Multiplying $e^{\lambda t}$ and then performing the operator ${_{a}D}_t^{\alpha-j-1}$ on both sides of
 Eq. \eqref{integral},
for $~0\leq j<n-2<n-1<\alpha<n$, we have
\begin{equation}\label{integralp}
\begin{split}
{_{a}D}_t^{\alpha-j-1}(e^{\lambda t}u(t))=&\sum_{k=0}^{n-2}g_k \frac{(t-a)^{j-k}}{\Gamma(j-k+1)}+{_{a}D}_t^{\alpha-j-1}{_{a}I}_t^{\alpha}\big(e^{\lambda t}  f(t,u(t))\big),
\\
=&\sum_{k=0}^{n-2}g_k \frac{(t-a)^{j-k}}{\Gamma(j-k+1)}+{_{a}D}_t^{-j-1}\big(e^{\lambda t}  f(t,u(t))\big),
\end{split}
\end{equation}
where  the formula
$$~{_{a}D}_t^{\alpha-j-1}\big((t-a)^{\alpha-k-1}\big)=
~{_{a}D}_t^{\alpha-j-1}\big((t-a)^{\alpha-k-1}\big)=\frac{\Gamma(\alpha-k)}{\Gamma(j-k+1)}(t-a)^{j-k},~0\leq k<n-2,$$
is utilized.

Taking a limit $t\rightarrow a$ in the above equation, we obtain
\begin{equation}\label{integralpv}
\lim _{t\rightarrow a}{_{a}D}_t^{\alpha-j-1}(e^{\lambda t}u(t))=\lim _{t\rightarrow a}\sum_{k=0}^{n-2}g_k \frac{(t-a)^{j-k}}{\Gamma(j-k+1)}+\lim _{t\rightarrow a}{_{a}D}_t^{-j-1}\big(e^{\lambda t}  f(t,u(t))\big),
\end{equation}
with the second term in the right hand side being equal to zero; and for its first term, we have
\begin{equation}\label{integralpvq}
\begin{split}
\lim _{t\rightarrow a}\sum_{k=0}^{n-2}g_k \frac{(t-a)^{j-k}}{\Gamma(j-k+1)}=&\lim _{t\rightarrow 0}\sum_{k=0}^{j-1}g_k \frac{(t-a)^{j-k}}{\Gamma(j-k+1)}+g_{j}+\lim _{t\rightarrow a}\sum_{k=j+1}^{n-2}g_k \frac{(t-a)^{j-k}}{\Gamma(j-k+1)}
\\
=&\sum_{k=0}^{j-1} \frac{g_k}{\Gamma(j-k+1)}\cdot0+g_{j}+\lim _{t\rightarrow a}\sum_{k=j+1}^{n-2} \frac{g_k(t-a)^{j-k}}{\infty}
\\
=&g_{j}.
\end{split}
\end{equation}
\end{proof}
By the similar technique in proving Lemma \ref{gldf}, we obtain the following conclusion
for the Cauchy problem \eqref{pRrlb}.
\begin{lemma}\label{caputodf}
If the function $f(t,u)$ is continuous, then $u(t)$ is the solution of the initial value problem \eqref{pRrlb}
if and only if $u(t)$ is the solution of the Volterra integral equation of the second kind
 \begin{equation}\label{integralcapa}
u(t)=\sum_{k=0}^{n-1}c_k\frac{e^{-\lambda t}(t-a)^{k}  }{\Gamma(k+1)}
+ \frac{1}{\Gamma(\alpha)}\int_{a}^t e^{-\lambda( t-s)} (t-s)^{\alpha-1} f(s,u(s))ds.
\end{equation}
In particular, if~ $0<\alpha<1$, then $u(t)$ satisfies the Cauchy problem if and only if
$u(t)$ satisfies the following integral equation
\begin{equation}\label{integralbz}
u(t)=u(a)e^{-\lambda (t-a) }
+ \frac{1}{\Gamma(\alpha)}\int_{a}^t e^{-\lambda( t-s)} (t-s)^{\alpha-1} f(s,u(s))ds.
\end{equation}
\end{lemma}

\subsection{Existence and  uniqueness}
Many authors have considered the existence and uniqueness of the solutions to the nonlinear ODEs with fractional derivatives; see, e.g., \cite{Pitcher:38,Abedeen:78,El-Sayed:88,Diethelm:02,Diethelm:04,Kilbas:06,Kilbas:08,Tomovski:12,Zhou:14}.
For the global existence and uniform asymptotic stability results
of fractional functional differential equations corresponding to \eqref{integralb}, one can see \cite{Li:13,Benchohra:08}.
In the following, we discuss the existence and uniqueness of the solutions of the nonlinear tempered fractional differential equations based on the equivalent Volterra equations presented in Lemmas \ref{gldf} and \ref{caputodf}.
We shall employ the Banach fixed point theorem to prove it.
Let $f : [a, b]\times B \rightarrow \mathbb{R}$ be a continuous function such that $f(t, u) \in L([a, b])$ for any $u\in B$, being an open  set in $\mathbb{R}$.
In the following, we always suppose that
$f(t, u)$ satisfies the Lipschitz type condition with respect
to the second variable
\begin{equation}\label{lipcd}
|f(t,u)-f(t,v)|\leq C_{Lip} |u-v|,\textrm{for all}~ u,v\in B,~t\in[a,b],
\end{equation}
where $C_{Lip}$ is constant. We shall use the following space
$$L^{\alpha,\lambda}([a,b])=\bigg\{u\in L([a,b]),\,{_{a}D}_t^{\alpha,\lambda}u(t)\in L([a,b])\bigg\}.$$
\begin{theorem}\label{thgldf}
If $n-1<\alpha<n, n \in \mathbb{N}^{+}, \lambda\geq 0$, 
then there exists a unique solution $u(t)$ to the Cauchy problem \eqref{pRrla} in the space $L^{\alpha,\lambda}([a,b])$.
\end{theorem}
\begin{proof}
The proof of this theorem is similar to the references \cite{Pitcher:38,Diethelm:02,Kilbas:06,Tomovski:12}.
First we prove the existence of a unique solution $u(t)\in L([a, b])$.
In accordance with Lemma \ref{gldf}, it is sufficient to prove
the existence of a unique solution $u(t)\in L([a, b])$ to the nonlinear Volterra integral equation \eqref{integral}.
 We rewrite the
integral equation \eqref{integral} in the form of  operator $$u(t)= (Pu)(t),$$ where
 \begin{equation}\label{Ta}
(Pu)(t)=u_0(t)+\frac{1}{\Gamma(\alpha)}\int_{a}^t e^{-\lambda( t-s)} (t-s)^{\alpha-1} f(s,u(s))ds,
\end{equation}
and
\begin{equation}\label{u0}
u_0(t)=\sum_{k=0}^{n-1}g_k\frac{e^{-\lambda t}(t-a)^{\alpha-k-1}}{\Gamma(\alpha-k)}.
\end{equation}
We first prove that $P$ is a contraction operator in the subinterval  $[a, t_1] \subset [a, b]~(a < t_1 < b)$.
To do this, we select $t_1\in ([a,b])$ such that the inequality
\begin{equation}\label{lipsa}
C_{Lip} ~\frac{(t_1-a)^{\alpha}}{\Gamma(\alpha+1)}<1
\end{equation}
holds.
 To apply the Banach fixed point theorem in the complete metric space $L([a, t_1])$, we have to prove the following facts:

(i) If $u(t) \in L([a, t_1])$, then $(Pu)(t) \in L([a, t_1])$;

(ii) For all $u_1,u_2 \in L([a, t_1])$ the following inequality holds
\begin{equation}\label{Clip}
\|Pu_1-Pu_2\|_{L(a,b)}\leq W_1 \|u_1-u_2\|_{L([a,b])}, ~W_1=C_{Lip} ~\frac{(t_1-a)^{\alpha}}{\Gamma(\alpha+1)}.
\end{equation}
In fact, since $f (t, u(t)) \in L([a, t_1])$ and Lemma \ref{lemcnim} in Appendix A, the integral in the right-hand side of \eqref{Ta} belongs to $L([a, t_1])$; obviously  $u_0(t) \in L([a, t_1])$, hence $(Pu)(t) \in L([a, t_1])$. Now, we prove the estimate \eqref{Clip}.
From Lemma \ref{lemcnim} in Appendix A,
we obtain
 \begin{equation*}
 \begin{split}
\|Pu_1-Pu_2\|_{L([a,t_1])}
=&\|{_{a}I_t^{\alpha,\lambda}}f(t,u_1(t))-{_{a}I_t^{\alpha,\lambda}}f(t,u_2(t))\|_{L([a,t_1])}\\
=&\|{_{a}I_t^{\alpha,\lambda}}\big(f(t,u_1(t))-f(t,u_2(t))\big)\|_{L([a,t_1])}\\
\leq& C_{Lip} \|{_{a}I_t^{\alpha,\lambda}}(u_1(t)-u_2(t))\|_{L([a,t_1])}\\
\leq& W_1 \|u_1(t)-u_2(t)\|_{L([a,t_1])}.
\end{split}
 \end{equation*}
In view of $0 < W_1 < 1$ there exists a unique solution $u^{*}(t)\in L([a, t_1])$ to Eq. \eqref{Ta} on the interval $[a, t_1]$. The solution $u(t)$ is obtained by taking the limit of convergent sequence $(P^mu^{*}_0)(t)$ as $m\rightarrow \infty$, i.e.,
\begin{equation}\label{Tmu}
\lim_{m\rightarrow \infty} \|P^mu^{*}_0-u^*\|_{L([a,t_1])}=0,
\end{equation}
where $u^{*}_0 (t) \in L([a, b])$.
Now let us consider the interval $[t_1, t_2]$ with $t_2 = t_1 + h$ and $h=t_1-a$. Rewrite
Eq. \eqref{integral} in the form
\begin{equation}\label{vuy}
  \begin{split}
u(t)=&u_0(t)+\frac{1}{\Gamma(\alpha)}\int_{a}^{t_1} e^{-\lambda( t-s)} (t-s)^{\alpha-1} f(s,u(s))ds\\
&+\frac{1}{\Gamma(\alpha)}\int_{t_1}^t e^{-\lambda( t-s)} (t-s)^{\alpha-1} f(s,u(s))ds.
 \end{split}
\end{equation}
Since the function $u(t)$ is uniquely defined on the interval $[a, t_1]$, the last integral can be considered as the known function. Then the above equation can be rewritten as
\begin{equation}\label{uy}
u(t)=u_{01}(t)
+\frac{1}{\Gamma(\alpha)}\int_{t_1}^t e^{-\lambda( t-s)} (t-s)^{\alpha-1} f(s,u(s))ds,
\end{equation}
where
\begin{equation}\label{iuy}
u_{01}(t)=u_0(t)+\frac{1}{\Gamma(\alpha)}\int_{0}^{t_1} e^{-\lambda( t-s)} (t-s)^{\alpha-1} f(s,u(s))ds,
\end{equation}
is the known function. With the same contraction factor $W_1$,  we can prove that there exists a unique solution $u^*(t) \in L(t_1, t_2)$ to Eq. \eqref{integral} on the interval $[t_1, t_2]$. By repeating this process finite times, e.g., $M$ times, we can cover the whole interval $[a,b]$.


To complete the proof of the theorem we must show that such a unique solution $u(t)\in L([a, b])$ belongs to the space $L^{\alpha,\lambda}([a,b])$. It is sufficient to
prove that ${_{a}D}_t^{\alpha,\lambda}\big(u(t)\big)\in L([a, b])$. By the above proof, the solution $u(t)\in L([a, b])$ is a limit of the sequence $u_{m}(t) \in L([a, b])$, i.e.,
\begin{equation}\label{Tmuau}
\lim_{m\rightarrow \infty} \|u_m-u\|_{L([a,b])}=0,
\end{equation}
with the choice of certain $u_{m}$ on each $[a, t_1], [t_1, t_2], \cdots , [t_{M-1}, b]$.
In view of
\begin{equation}\label{caputox}
\|_{a}D_t^{\alpha,\lambda}u_m-{_{a}D_t^{\alpha,\lambda}}u\|_{L([a, b])}=\|f(t,u_{m-1})-f(t,u)\|_{L([a, b])}\leq C_{Lip} \|u_{m-1}-u\|_{L([a, b])},
\end{equation}
taking the limit of \eqref{caputox} as $m\rightarrow\infty$, gives
\begin{equation}
\lim_{m\rightarrow\infty}\|_{a}D_t^{\alpha,\lambda}u_m-{_{a}D}_t^{\alpha,\lambda}u\|_{L([a, b])}=0,
\end{equation}
and  hence $_{a}D_t^{\alpha,\lambda}u(t)\in L([a,b])$. This completes the proof of the theorem.
\end{proof}

By almost the same idea, we can prove the following existence and uniqueness result for the Cauchy type problem \eqref{pRrlb}.
\begin{theorem}\label{thgldfcc}
If $n-1<\alpha<n, \,n \in \mathbb{N}^{+}, \lambda\geq 0$, then
there exists a unique solution $u(t)$ to the Cauchy type problem \eqref{pRrlb} in the space $AC^n[a, b]$.
\end{theorem}
\subsection{Stability}\label{gcpc}
To prove the stability of the solutions of the Cauchy type problems \eqref{pRrla} and \eqref{pRrlb}, we need the following generalized Gronwall's Lemmas.
\begin{lemma}[\cite{Coddington:55}]\label{gronwall}
Let $x,y,\phi$ be real continuous functions  on interval $[a,b]$,  $\phi(t)\geq 0, t\in[a,b]$, and
\begin{equation}\label{condition}
x(t)\leq y(t)+\int_{a}^{t}x(s)\phi(s)ds.
\end{equation}
Then holds
\begin{equation}\label{conloasa}
x(t)\leq y(t)+\int_{a}^{t}y(s)\phi(s)e^{\int_{s}^{t}\phi(\tau)d\tau}ds.
\end{equation}
If, in addition, $y(\cdot)$ is a nondecreasing function defined on $[a,b]$, we have
\begin{equation}\label{conlo}
x(t)\leq y(t)e^{\int_{a}^{t}\phi(s)ds}.
\end{equation}
\end{lemma}
\begin{lemma}[\cite{Brunner:86}]\label{gronwalla}
Let $x:[a,b]\rightarrow [0,+\infty)$ be a real function and $y(\cdot)$ is a nonnegative, locally integrable function defined on $[a,b]$, $\alpha\in(0,1)$, and there exists a constant $L>0$ such that
\begin{equation}\label{conditiona}
x(t)\leq y(t)+L\int_{a}^{t}x(s)(t-s)^{\alpha-1}ds.
\end{equation}
Then there exists a constant $C=C(\alpha)$ such that
\begin{equation}\label{conlo}
x(t)\leq y(t)+C\int_{a}^{t}y(s)(t-s)^{\alpha-1}ds,
\end{equation}
for all $t\in [a,b]$.
\end{lemma}
\begin{theorem}\label{thgldfsta}
Under the  assumptions given in Theorem \ref{thgldf}, let $u(t)$ and $v(t)$ be the solutions of the Cauchy type problem \eqref{pRrla} with different initial conditions. Then
\begin{equation}\label{conloa}
|u(t)-v(t)|\leq
\begin{cases}
\displaystyle
\frac{\big|g_0-\widetilde{g}_0\big|e^{-\lambda t}}{\Gamma(\alpha)}\big((t-a)^{\alpha-1}
+C(\alpha)
(t-a)^{2\alpha-1}\big),~ \text{$0<\alpha<1$},\\
\displaystyle
\varphi(t)
\left(e^{\frac{(t-a)^{\alpha}}{\alpha}}(t-a)^\alpha+1 \right)e^{-\lambda t},~\text{$n-1<\alpha<n,n\geq2$},
\end{cases}
\end{equation}
where  $\displaystyle\varphi(t)=C(\alpha)\sum_{k=0}^{n-1}\big|g_k-\widetilde{g}_k\big|(t-a)^{\alpha-k-1}$,
$g_k=\big[{_{a}D}_t^{\alpha-k-1}\big(e^{\lambda x}u(t)\big)\big]\big|_{t=a}$, and

 $
\widetilde{g}_k=\big[{_{a}D}_t^{\alpha-k-1}\big(e^{\lambda x}v(t)\big)\big]\big|_{t=a}$.

Similarly, under the  assumptions given in Theorem \ref{thgldfcc}, for the problem  \eqref{pRrlb}, there exists
\begin{equation}\label{conlob}
|u(t)-v(t)|\leq
\begin{cases}
\displaystyle
\big|c_0-\widetilde{c}_0\big|e^{-\lambda t}\bigg(1+C\frac{(t-a)^{\alpha}}{\alpha}\bigg),~ \text{$0<\alpha<1$},\\
\displaystyle
\psi(t)
\left(e^{\frac{(t-a)^{\alpha}}{\alpha}}(t-a)^\alpha+1 \right)e^{-\lambda t},~\text{$n-1<\alpha<n,n\geq2$},
\end{cases}
\end{equation}
where $\displaystyle\psi(t)=C(\alpha)\sum_{k=0}^{n-1}\big|c_k-\widetilde{c}_k\big|(t-a)^{k},
c_k=\bigg[\frac{d^{k}}{dt^{k}}(e^{\lambda t}u(t))\bigg]\bigg|_{t=a}$, and $\widetilde{c}_k=\bigg[\frac{d^{k}}{dt^{k}}(e^{\lambda t}v(t))\bigg]\bigg|_{t=a}$.
\end{theorem}
\begin{proof}
Suppose that $u(t)$ and $v(t)$ are any two solutions of the Cauchy type problem \eqref{pRrla} with different initial conditions.
From the equivalent integral formulation (\ref{integral}), we have
\begin{equation}\label{pvuys}
  \begin{split}
e^{\lambda t}|u(t)-v(t)|\leq&\left(\sum_{k=0}^{n-1}\big|g_k-\widetilde{g}_k\big|\frac{(t-a)^{\alpha-k-1}}{\Gamma(\alpha-k)}\right)\\
&+\frac{1}{\Gamma(\alpha)}\int_{a}^t  (t-s)^{\alpha-1} e^{\lambda s}|f(s,u(s))-f(s,v(s))|ds\\
&\leq\left(\sum_{k=0}^{n-1}\big|g_k-\widetilde{g}_k\big|\frac{(t-a)^{\alpha-k-1}}{\Gamma(\alpha-k)}\right)\\
&+\frac{C_{lip}}{\Gamma(\alpha)}\int_{a}^t  (t-s)^{\alpha-1} e^{\lambda s}|u(s)-v(s)|ds.
 \end{split}
\end{equation}
For $0<\alpha<1$, using the generalized Gronwall's inequality with weak singular kernel given in Lemma \ref{gronwalla}, we get
\begin{equation}\label{conlovcc}
\begin{split}
e^{\lambda t}|u(t)-v(t)|\leq& \bigg(\big|g_0-\widetilde{g}_0\big|\frac{(t-a)^{\alpha-1}}{\Gamma(\alpha)}\bigg)\\
&+C\int_{a}^{t}\bigg(\big|g_0-\widetilde{g}_0\big|\frac{(s-a)^{\alpha-1}}{\Gamma(\alpha)}\bigg)(t-s)^{\alpha-1}ds,
\end{split}
\end{equation}
which implies
\begin{eqnarray}\label{conlovcsc}
|u(t)-v(t)| & \leq & \big|g_0-\widetilde{g}_0\big|\frac{e^{-\lambda t}}{\Gamma(\alpha)}\bigg((t-a)^{\alpha-1}+C\int_{a}^{t}(s-a)^{\alpha-1}(t-s)^{\alpha-1}ds\bigg)\nonumber\\
& = & \big|g_0-\widetilde{g}_0\big|\frac{e^{-\lambda t}}{\Gamma(\alpha)}\bigg((t-a)^{\alpha-1}+C\frac{\Gamma(\alpha)^2}{\Gamma(2\alpha)}(t-a)^{2\alpha-1}\bigg).
\end{eqnarray}
For $n-1<\alpha< n,n\geq2$, taking $$x(t)=e^{\lambda t}|u(t)-v(t)|, ~ y(t)=\left(\sum_{k=0}^{n-1}\big|g_k-\widetilde{g}_k\big|\frac{(t-a)^{\alpha-k-1}}{\Gamma(\alpha-k)}\right),~ \phi(t)= (t-s)^{\alpha-1} $$ in inequality \eqref{condition}, we have
\begin{eqnarray}\label{conlovccn}
&& e^{\lambda t}|u(t)-v(t)| \nonumber\\
&\leq &  \sum_{k=0}^{n-1}\big|g_k-\widetilde{g}_k\big|\frac{(t-a)^{\alpha-k-1}}{\Gamma(\alpha-k)}
+\sum_{k=0}^n |g_k-\tilde{g}_k|(t-a)^{2\alpha-k-1}\frac{\Gamma(\alpha)}{\Gamma(2\alpha-k)}e^{\frac{(t-a)^{\alpha}}{\alpha}} \nonumber\\
&\leq & C(\alpha) \sum_{k=0}^{n-1}\big|g_k-\widetilde{g}_k\big|(t-a)^{\alpha-k-1}\left(e^{\frac{(t-a)^{\alpha}}{\alpha}}(t-a)^\alpha+1 \right)e^{-\lambda t},
\end{eqnarray}
where $C(\alpha)=\max\limits_{k=0,1,\cdots,n-1}\left\{\frac{1}{\Gamma(\alpha-k)}+\frac{\Gamma(\alpha)}{\Gamma(2\alpha-k)}\right\}$.


With the similar method, we can prove the stability results for the problem \eqref{pRrlb}.
\end{proof}
\subsection{generalized cauchy problems}\label{gcp}
In this subsection, we consider the ODE with $m$-term Riemann-Liouville tempered fractional derivatives
\begin{equation}\label{pRrlaag}
\begin{cases}
\displaystyle
_{a}D_t^{\alpha,\lambda}u(t)=
f(t,u(t),{_{a}D}_t^{\alpha_{1},\lambda_{1}}u(t),{_{a}D}_t^{\alpha_{2},\lambda_{2}}u(t),\cdots,{_{a}D}_t^{\alpha_{m-1},\lambda_{m-1}}u(t)),~ \text{$n-1<\alpha<n$},\\
\displaystyle
\big[{_{a}D}_t^{\alpha-k-1}\big(e^{\lambda x}u(t)\big)\big]\big|_{t=a}=g_k,~\text{$k=0,1,2,\cdots,n-1$};
\end{cases}
\end{equation}
and the ODE with $m$-term  Caputo tempered fractional derivatives
\begin{equation}\label{pRrlbbg}
\begin{cases}
\displaystyle
 {_{a}^{C}D}_t^{\alpha,\lambda}u(t)=
 f(t,u(t),{_{a}^{C}D}_t^{\alpha_1,\lambda_1}u(t),{_{a}^{C}D}_t^{\alpha_2,\lambda_2}u(t),\cdots,{_{a}^{C}D}_t^{\alpha_{m-1},\lambda_{m-1}}u(t)),~ \text{$n-1<\alpha<n$},\\
\displaystyle
\bigg[\frac{d^{k}}{dt^{k}}(e^{\lambda t}u(t))\bigg]\bigg|_{t=a}=c_k,~\text{$k=0,1,2,\cdots,n-1$},
\end{cases}
\end{equation}
where  $\lambda_{i}\geq0$, the real value $\alpha \in (n-1,n),n\in \mathbb{N}^{+},i=1,2,\cdots, m-1$, and
 $0<\alpha_1<\alpha_2<\cdots<\alpha_{m-1}<\alpha,\,m\geq2$.
Similar to the Cauchy problems \eqref{pRrla} and \eqref{pRrlb}, we have the following lemma.
\begin{lemma}\label{gldfgen}
 Assume that $B$ is an open set in $\mathbb{R}^m$
 and $f:(a,b)\times B\rightarrow \mathbb{R}$ is a function such that $f(t,u,u_1,u_2,...,u_{m-1})\in L([a,b])$ for all $(u,u_1,u_2,..,u_{m-1})\in B$.

(1) If  $u(t)\in L([a,b])$ is the solution of the initial value problem \eqref{pRrlaag}
if and only if $u(t)$ is the solution of the Volterra integral equation of the second kind
\begin{equation}\label{integralgen}
  \begin{split}
&u(t)
\\
=&\sum_{k=0}^{n-1}g_k\frac{e^{-\lambda t}(t-a)^{\alpha-k-1}}{\Gamma(\alpha-k)}\\
& + \frac{1}{\Gamma(\alpha)}\int_{a}^t e^{-\lambda( t-s)} (t-s)^{\alpha-1} f(s,u(s),{_{a}D}_s^{\alpha_1,\lambda_1}u(s),{_{a}D}_s^{\alpha_2,\lambda_2}u(s),...,{_{a}D}_s^{\alpha_{m-1},\lambda_{m-1}}u(s))ds.
 \end{split}
\end{equation}
(2) If $f(t,u,u_1,u_2,...,u_{m-1})$ is continous, then $u(t)$ is the solution of the initial value problem \eqref{pRrlbbg}
if and only if $u(t)$ is the solution of the Volterra integral equation of the second kind
 \begin{equation*}
 \begin{split}
& u(t)\\
=&\sum_{k=0}^{n-1}c_k\frac{e^{-\lambda t}(t-a)^{k}  }{\Gamma(k+1)}\\
&+ \frac{1}{\Gamma(\alpha)}\int_{a}^t e^{-\lambda( t-s)} (t-s)^{\alpha-1} f(s,u(s),{_{a}^CD}_s^{\alpha_1,\lambda_1}u(s),{_{a}^CD}_s^{\alpha_2,\lambda_2}u(s),...,{_{a}^CD}_s^{\alpha_{m-1},\lambda_{m-1}}u(s))ds.
 \end{split}
\end{equation*}
\end{lemma}

Suppose that $f(t,u_1,u_2,...,u_{m})$ is a continuous function satisfying the Lipschitz type condition
\begin{equation}\label{mlipz}
|f(t,u_1,u_2,...,u_m)-f(t,v_1,v_2,...,v_{m})|\leq C_{Lipg} \sum_{j=1}^{m}|u_j-v_j|
\end{equation}
for all $t\in[a,b]$ and $(u_1,u_2,...,u_{m}),\, (v_1,v_2,...,v_{m})\in B$, where the Lipschitz constant $C_{Lipg}$ does not depend on $t\in[a,b]$.
\begin{theorem}\label{thgldfgen}
Let $B$ be an open  set in $\mathbb{R}^m$ and $f : [a, b]\times B \rightarrow \mathbb{R}$ be a function such that $f(t, u) \in L([a, b])$ for any $u\in B$ and be Lipschitz continuous.

(1) 
There exists a unique solution $u(t)$ to the Cauchy type problem \eqref{pRrlaag} in the space $L^{\alpha,\lambda}([a,b])$.

(2) 
There exists a unique solution $u(t)$ to the Cauchy type problem \eqref{pRrlbbg} in the space $AC^n[a, b]$.
\end{theorem}
\begin{proof}
Similar to Theorems \ref{thgldf} and \ref{thgldfcc}. We begin our proof from the integral equations given in Lemma \ref{gldfgen}. We only prove the  Cauchy type problem \eqref{pRrlaag}.
Let $t_1$ belong to $(a,b)$ 
such that the inequality
$$C_{Lipg}~\sum_{j=0}^{m-1}\frac{(t_1-a)^{\alpha-\alpha_j}}{\Gamma(\alpha-\alpha_j+1)}<1,$$
holds.
  The operator corresponding to \eqref{integralgen} takes the form
 \begin{equation}\label{gTa}
(Tu)(t)=u_0(t)+\frac{1}{\Gamma(\alpha)}\int_{a}^t e^{-\lambda( t-s)} (t-s)^{\alpha-1} f(s,u(s),{_{a}D}_s^{\alpha_1,\lambda_1}u(s),...,{_{a}D}_s^{\alpha_{m-1},\lambda_{m-1}}u(s))ds,
\end{equation}
where
\begin{equation}\label{gu0}
u_0(t)=\sum_{k=0}^{n-1}g_k\frac{e^{-\lambda t}(t-a)^{\alpha-k-1}}{\Gamma(\alpha-k)}.
\end{equation}

From the Lipschitz condition \eqref{mlipz} it directly follows that
 \begin{equation*}
 \begin{split}
&|{_{a}I_t^{\alpha,\lambda}}f(t,u(t),({_{a}D}_t^{\alpha_1,\lambda_1}u)(t),...,({_{a}D}_t^{\alpha_{m-1},\lambda_{m-1}}u)(t)) \\
&-{_{a}I_t^{\alpha,\lambda}}f(t,v(t),({_{a}D}_t^{\alpha_1,\lambda_1}v)(t),...,({_{a}D}_t^{\alpha_{m-1},\lambda_{m-1}}v)(t))| \\\
&\leq C_{Lipg}~ {_{a}I_t^{\alpha,\lambda}}\bigg(\bigg|\sum_{j=0}^{m-1}{_{a}D}_t^{\alpha_j,\lambda_j}(u-v)\bigg|\bigg)(t)\\
&\leq C_{Lipg}~ \sum_{j=0}^{m-1}\bigg({_{a}I_t^{\alpha-\alpha_j,\lambda}}\big|{I_t^{\alpha_j,\lambda}}{_{a}D}_t^{\alpha_j,\lambda_j}(u(t)-v(t))\big|\bigg).
\end{split}
 \end{equation*}
Furthermore, using the composition formula \eqref{rca}, we have
\begin{equation*}
 \begin{split}
&|{_{a}I_t^{\alpha,\lambda}}f(t,u(t),{_{a}D}_t^{\alpha_1,\lambda_1}u(t),...,{_{a}D}_t^{\alpha_{m-1},\lambda_{m-1}}u(t))
\\
&-{_{a}I_t^{\alpha,\lambda}}f(t,v(t),{_{a}D}_t^{\alpha_1,\lambda_1}v(t),...,{_{a}D}_t^{\alpha_{m-1},\lambda_{m-1}}v(t))| \\\
&\leq C_{Lipg}~ \sum_{j=0}^{m-1}\bigg({_{a}I_t^{\alpha-\alpha_j,\lambda}}\bigg|(u-v)
\\
&~~ -\sum_{k_{j}=0}^{n_{j}-1}\frac{e^{-\lambda t}(t-a)^{\alpha_j-k_{j}-1}}{\Gamma(\alpha_j-k_{j})}\big[{_{a}D}_t^{\alpha_j-k_{j}-1}(e^{\lambda t}(u(t)-v(t)))\big|_{t=a}\big]\bigg|\bigg)(t),
\end{split}
 \end{equation*}
where $n_j$ is the smallest integer larger than or equal to $\alpha_j$.

From the given initial conditions, it can be checked that  ${_{a}D}_t^{\alpha_j-k_{j}-1}(e^{\lambda t}(u(t)-v(t)))\big|_{t=a}=0$. Then
\begin{equation}
 \begin{split}
&|{_{a}I_t^{\alpha,\lambda}}f(t,u(t),{_{a}D}_t^{\alpha_1,\lambda_1}u(t),...,{_{a}D}_t^{\alpha_{m-1},\lambda_{m-1}}u(t))
\\
& -{_{a}I_t^{\alpha,\lambda}}f(t,v(t),{_{a}D}_t^{\alpha_1,\lambda_1}v(t),...,{_{a}D}_t^{\alpha_{m-1},\lambda_{m-1}}v(t))| \\\
&\leq C_{Lipg}~ \sum_{j=0}^{m-1}\bigg({_{a}I_t^{\alpha-\alpha_j,\lambda}}\big|(u(t)-v(t))\big|\bigg),
\end{split}
 \end{equation}
 for any  $t\in[a,b]$.
Taking $t=t_1$ in above the formula and applying \eqref{TCDDb}, we get
\begin{equation*}
\|Tu_1-Tu_2\|_{L([a,t_1])}
\leq C_{Lipg} K \|u_1(t)-u_2(t)\|_{L([a,t_1])},K=C_{Lipg} ~\sum_{j=0}^{m-1}\frac{(t_1-a)^{\alpha-\alpha_j}}{\Gamma(\alpha-\alpha_j+1)}.
 \end{equation*}
It follows that there exists a unique solution  $u^{*}(t)$ to Eq. \eqref{integralgen} in $L([a,t_1])$.
This solution is obtained as a limit of the convergent sequence $(T^{j}u^{*}_0)(t)=u_{j}(t)$, and holds
\begin{equation}\label{Taitea}
\lim_{j\rightarrow \infty} \|T^{j}u^{*}_0-u^*\|_{L([a,t_1])}=0,
\end{equation}
i.e.,
\begin{equation}\label{Tmua}
\lim_{j\rightarrow \infty} \|u_j-u^*\|_{L([a,t_1])}=0.
\end{equation}
With the same fashion of proving Theorem \ref{thgldf}, we can show that there exists an unique solution $u(t)\in L^{\alpha,\lambda}([a,b])$
to Eq. \eqref{integralgen}.
In addition,
\begin{equation*}
 \begin{split}
&\|_{a}D_t^{\alpha,\lambda}u_j-{_{a}D_t^{\alpha,\lambda}}u\|_{L([a,b])}\\
&=\|f(t,u_{j-1},{_{a}D}_t^{\alpha_1,\lambda_1}u_{j-1},...,{_{a}D}_t^{\alpha_{m-1},\lambda_{m-1}}u_{j-1})
\\
&
~~-f(t,u,{_{a}D}_t^{\alpha_1,\lambda_1}u,...,{_{a}D}_t^{\alpha_{m-1},\lambda_{m-1}}u)\|_{L([a,b])}\\
&\leq K\|u_{j-1}-u\|_{L([a,b])}\rightarrow 0,~~j\rightarrow\infty,
 \end{split}
\end{equation*}
which implies that $_{a}D_t^{\alpha,\lambda}u(t)\in L([a,b])$.
\end{proof}

\section{Numerical algorithm for the tempered fractional ODE}\label{sec:algorithm}
The well-posedness of the tempered fractional ODE has been carefully discussed in the above sections. Usually, it is hard to find the analytical solutions of the tempered fractional ODE, especially for the nonlinear case. Efficient numerical algorithm naturally becomes an urgent topic for this type of equation. Now, we extend the so-called Jacobi-predictor-corrector algorithm \cite{Zhao:13} to the the tempered fractional ODE; and its striking benefits are still kept, including having any desired convergence orders and the linearly increasing computational cost with the time $t$.


\subsection{Jacobi-predictor-corrector algorithm}
The equations \eqref{integral} and \eqref{integralcapa} can be rewritten as the following  Volterra integral equation
\begin{equation}\label{numericalproblem}
u(t)=e^{-\lambda t}\sum_{k=0}^{n-1}a_k(t)
+ \frac{e^{-\lambda t}}{\Gamma(\alpha)}\int_{a}^t (t-s)^{\alpha-1} g(s,u(s))ds,
\end{equation}
where  $a_k(t)=c_k\frac{(t-a)^{k}}{\Gamma(k+1)}$ for tempered Caputo derivative, and $a_k(t)=g_k\frac{(t-a)^{\alpha-k-1}}{\Gamma(\alpha-k)}$ for tempered Riemann-Liouville derivative, and $g(t,u(t))=e^{\lambda t} f(t,u(t))$.


Firstly, we transform the original equation into
\begin{eqnarray}\label{numer1}
u(t)&=&e^{-\lambda t}\sum_{k=0}^{n-1}a_k(t)
+ \frac{e^{-\lambda t}}{\Gamma(\alpha)}\int_{a}^t (t-s)^{\alpha-1} g(s,u(s))ds
\nonumber\\
&=&e^{-\lambda t}\sum_{k=0}^{n-1}a_k(t)+\frac{e^{-\lambda t}}{\Gamma(\alpha)}\bigg(\frac{t-a}{2}\bigg)^\alpha
\int^{1}_{-1}(1-z)^{\alpha-1}\tilde{g}\big(z,\tilde{u}(z)\big)dz,
\end{eqnarray}
where
\begin{equation*}
\tilde{g}\big(z,\tilde{u}(z)\big)=g\bigg(\frac{t-a}{2}z+\frac{t+a}{2},u\left(\frac{t-a}{2}z+\frac{t+a}{2}\right)\bigg),~~-1\leq z \leq 1;
\end{equation*}
\begin{equation*}
\tilde{u}(z)=u\bigg(\frac{t-a}{2}z+\frac{t+a}{2}\bigg), ~~-1\leq z \leq 1.
\end{equation*}
Using $(N +1)$-point Jacobi-Gauss-Lobatto
quadrature to approximate  the
integral in \eqref{numer1}  yields
\begin{equation}\label{numer2}
u(t)\approx e^{-\lambda t}\sum_{k=0}^{n-1}a_k(t)+\frac{e^{-\lambda t}}{\Gamma(\alpha)}\bigg(\frac{t-a}{2}\bigg)^\alpha
\sum_{j=0}^{N}\omega_{j}~\tilde{g}\big(z_j,\tilde{u}(z_j)\big),
\end{equation}
where we choose  $\omega(z)=(1-z)^{\alpha-1}(1+z)^0$ as the weight function; $\{z_j\}_{j=0}^{N}$ and $\{\omega_{j}\}_{j=0}^{N}$
are the $(N + 1)$-degree Jacobi-Gauss-Lobatto nodes and their corresponding
Jacobi weights in the reference interval
$[-1,1]$, respectively; see, e.g., \cite{Shen:11}.

Now we turn to describe the computational scheme for Eq. \eqref{numer1}. For this purpose,
we define a grid in the interval $[a,b]$ with $M+1$ equidistant
nodes $t_j$, given by
\begin{equation}\label{numer3}
t_j=j\tau+a,~~~j=0,\cdots,M,
\end{equation}
where $\tau=(b-a)/M$ is the stepsize. Suppose that we have got the numerical values of $u(t)$ at
$t_0,t_1,\cdots,t_n$, which are denoted as $u_0,u_1,\cdots,u_n$; now
we are going to compute the value of $u(t)$ at $t_{n+1}$, i.e., $u_{n+1}$.
From Eq. (\ref{numer2}), we have
\begin{eqnarray}\label{numer4}
&&u(t_{n+1})
\nonumber\\
&\approx & e^{-\lambda t_{n+1}}\sum_{k=0}^{n-1}a_k(t_{n+1})
+\frac{e^{-\lambda t_{n+1}}}{\Gamma(\alpha)}\bigg(\frac{t_{n+1}-a}{2}\bigg)^\alpha
\sum_{j=0}^{N}\omega_{j}~\tilde{g}_{n+1}\big(z_j,\tilde{u}_{n+1}(z_j)\big)
\nonumber\\
&= & e^{-\lambda t_{n+1}}\sum_{k=0}^{n-1}a_k(t_{n+1})
+\frac{1}{\Gamma(\alpha)}\bigg(\frac{t_{n+1}-a}{2}\bigg)^\alpha
\sum_{j=0}^{N}\omega_{j}~e^{\frac{t_{n+1}-a}{2}\lambda(z_j-1)}\tilde{f}_{n+1}\big(z_j,\tilde{u}_{n+1}(z_j)\big),
\end{eqnarray}
where
\begin{equation*}
\tilde{f}_{n+1}\big(z,\tilde{u}_{n+1}(z)\big)=f\bigg(\frac{t_{n+1}-a}{2}z+\frac{t_{n+1}+a}{2},u\left(\frac{t_{n+1}-a}{2}z+\frac{t_{n+1}+a}{2}\right)\bigg),~~-1\leq z \leq 1,
\end{equation*}
\begin{equation*}
\tilde{u}_{n+1}(z)=u\bigg(\frac{t_{n+1}-a}{2}z+\frac{t_{n+1}+a}{2}\bigg), ~~-1\leq z \leq 1.
\end{equation*}

To compute the second summation term of \eqref{numer4},  we need
to evaluate the value of $f$ at the point $t_{n+1}$ due to
$\tilde{f}_{n+1}\big(z_{N},\tilde{u}_{n+1}(z_{N})\big)=f\big(t_{n+1},u(t_{n+1})\big)$.
It can be numerically approximated by using the piecewise linear interpolation of the term $f$ \cite{Diethelm:02,Diethelm:04,Deng:07}. Here, we do it using the technique given in our previous work
\cite{Zhao:13}. More concretely, we interpolate the function $f$ by the known ``neighborhood" points of $t_{n+1}$.
 For the other values
of $\tilde{f}_{n+1}\big(z_j,\tilde{u}_{n+1}(z_j)\big),\,0 \leq j
\leq N-1$, we can also obtain them based on the interpolation of
$f$ at the time nodes located in the ``neighborhood" points of $z_j$
(should be $(1+z_j)t_{n+1}/2$ as to variable $t$). Denote $N_I$ as
the number of time nodes used for the interpolation. In  practical application, we use the improved
predictor-corrector formulas given in \cite{Deng:07} to get the values at $t_0,t_1,\cdots,t_{N_I-1}$
as the known `initial' values. For the criterion of choosing the ``neighborhood" points, refer to \cite{Zhao:13}.


Collecting the above analysis, we get the predictor-corrector formulas of
(\ref{numericalproblem}) as
\begin{eqnarray}\label{pre}
&&u_{n+1}
\nonumber\\
&=&e^{-\lambda t_{n+1}}\sum_{k=0}^{n-1}a_k(t_{n+1})
\nonumber\\
&&+\frac{1}{\Gamma(\alpha)}\bigg(\frac{t_{n+1}-a}{2}\bigg)^\alpha
\Big(\sum_{j=0}^{N-1}\omega_{j}e^{\frac{t_{n+1}-a}{2}\lambda(z_j-1)}\tilde{f}_{n+1}\big(z_j,\tilde{u}_{n+1}(z_j)\big)
+\omega_{N}f(t_{n+1},u^P_{n+1})\Big),
\end{eqnarray}
and
\begin{equation}\label{cor}
u^P_{n+1}=e^{-\lambda t_{n+1}}\sum_{k=0}^{n-1}a_k(t_{n+1})
+\frac{1}{\Gamma(\alpha)}\bigg(\frac{t_{n+1}-a}{2}\bigg)^\alpha
\sum_{j=0}^{N}\omega_{j}e^{\frac{t_{n+1}-a}{2}\lambda(z_j-1)}\tilde{f}^P_{n+1}\big(z_j,\tilde{u}_{n+1}(z_j)\big),
\end{equation}
where $\{\tilde{f}^P_{n+1,j}\}_{j=0}^{N}$ in (\ref{cor}) means that all the values of
$\tilde{f}_{n+1}$ at the Jacobi-Gauss-Lobatto nodes are got by using the interpolations based on
the values of $\{f(t_i,u_i)\}_{i=0}^n$; whereas $\{\tilde{f}_{n+1,j}\}_{j=0}^{N-1}$ in
(\ref{pre}) are obtained by using the interpolations based on the values of
$\{f(t_j,u_j)\}_{j=0}^n$ and $f(t_{n+1},u^P_{n+1})$. From the computational scheme (\ref{pre})-(\ref{cor}), it can be clearly seen that the computational cost linearly increase with $n$ (or time $t$).

With the similar methods given in \cite{Zhao:13}, we can get the following estimates 
for the  Volterra integral system \eqref{numericalproblem}.
\begin{theorem}\label{thgldfccr}
If $g(t,u(t))$ is Lipschitz continuous with respect to the second variable, and has the form
$$g(t,u(t))=\sum_{k=1}^{m}t^{\mu_k}w_k(t)+\delta,$$
where $w_k(t), 1\leq N_I\leq\mu_1\leq...\leq \mu_m$ are sufficiently smooth and $m$ can be $+\infty$, $\delta$ is a constant,
 then there exists a constant  $C$ being independent of $n,\tau,N$, such that
$$\max_{1\leq n+1\leq M}\big|u(t_{n+1})-u_{n+1}\big|\leq C\tau^{N_I}.$$  
\end{theorem}
This theorem shows that the scheme (\ref{pre})-(\ref{cor}) potentially have any desired convergence order by adjusting the number of interpolation points $N_I$.

\subsection{Numerical test}
In this subsection, we consider two simple numerical examples to show
the numerical errors and convergence orders of the Jacobi-predictor-corrector method. The two examples are Caputo tempered ODEs; solving the Riemann-Liouville tempered ODEs can be done in the same way, so is omitted here.
\begin{example}\label{example2}
Consider the Caputo tempered factional
initial value problem
\begin{equation}\label{rfrodes2}
_{0}^{C}D_{t}^{\alpha,\lambda}u(t)=e^{-\lambda t}
\bigg(\frac{\Gamma(9)}{\Gamma(9-\alpha)}t^{8-\alpha}+t^8+\frac{9}{4}t^\alpha+\frac{9}{4}\Gamma(\alpha+1)\bigg)-u(t).
\end{equation}
The initial values are chosen as $u(t)|_{t=0}=0$ and $\big[\frac{d}{dt}(e^{\lambda t}u(t))\big]\big|_{t=0}=0$ for $1<\alpha<2$, and as $u(t)|_{t=0}=0$ for $0<\alpha<1$.
Using the formula
\begin{equation}\label{formula}
{_0}D_t^{\alpha,\lambda}\big[e^{-\lambda t}t^{\mu}\big]=\frac{\Gamma(\mu+1)}{\Gamma(\mu-\alpha+1)}e^{-\lambda t}t^{\mu-\alpha},
\end{equation} it can be checked that the exact solution of this initial value problem is
$$u(t)=e^{-\lambda t}\bigg(t^8+\frac{9}{4}t^{\alpha}\bigg).$$
\end{example}
\begin{table}[h]
\centering
\caption{Maximum errors and convergence orders of Example \ref{example2} solved by the scheme (\ref{pre})-(\ref{cor}) with $T=1,N=20,N_I=7$, and  $\alpha=0.5$.}
\vspace{0.2cm}
  \begin{tabular}{|c|c|c|c|c|c|c|c|c|}
\hline
 &\multicolumn{2}{c|}{$\lambda=0$} & \multicolumn{2}{c|}{$\lambda=2$} & \multicolumn{2}{c|}{$\lambda=6$}\\
\cline{2-3} \cline{3-4}\cline{4-7}
 $\tau$ & error  & order &  error & order & error & order\\
\hline
 1/10  & 1.5207e-004 &        &2.3516e-005 &         &1.4300e-006 &   \\
 1/20  & 4.6202e-007 & 8.3626 &1.4040e-007 & 7.3879  &3.3507e-008 & 5.4154\\
 1/40  & 1.6877e-009 & 8.0967 &6.3106e-010 & 7.7976  &2.7846e-010 & 6.9109 \\
 1/80  & 8.1135e-012 & 7.7005 &2.5491e-012 & 7.9517  &1.4371e-012 & 7.5982\\
 1/160 & 3.5305e-014 & 7.8443 &1.2794e-014 & 7.6383  &7.0913e-015 & 7.6629\\
   \hline
  \end{tabular}
\label{tabexm1}
\end{table}
\begin{table}[h]
\centering
\caption{Maximum errors and convergence orders of Example \ref{example2} solved by the scheme (\ref{pre})-(\ref{cor}) with $T=1,N=20,N_I=6$, and $\alpha=1.0$.}
\vspace{0.2cm}
  \begin{tabular}{|c|c|c|c|c|c|c|c|c|}
\hline
 &\multicolumn{2}{c|}{$\lambda=0$} & \multicolumn{2}{c|}{$\lambda=2$} & \multicolumn{2}{c|}{$\lambda=6$}\\
\cline{2-3} \cline{3-4}\cline{4-7}
 $\tau$ & error  & order &  error & order & error & order\\
\hline
1/10  & 8.1108e-005 &        &1.2528e-005&          &1.1365e-006 & \\
 1/20  & 7.8788e-007 & 6.6857 &1.5673e-007 & 6.3207  &2.3299e-008 & 5.6082  \\
 1/40  & 1.2817e-008 & 5.9418 &2.1909e-009 & 6.1606  &3.2657e-010 & 6.1567\\
 1/80  & 2.2418e-010 & 5.8373 &3.4124e-011 & 6.0046  &4.4768e-012 & 6.1888 \\
 1/160 & 3.6193e-012 & 5.9528 &5.3461e-013 & 5.9962  &6.7955e-014 & 6.0417\\ 
  \hline
  \end{tabular}
\label{tabexm2}
\end{table}
\begin{table}[h]
\centering
\caption{Maximum errors and convergence orders of Example \ref{example2} solved by the scheme (\ref{pre})-(\ref{cor}) with $T=1,N=20,N_I=6$, and $\alpha=1.5$.}
\vspace{0.2cm}
  \begin{tabular}{|c|c|c|c|c|c|c|c|}
\hline
&\multicolumn{2}{c|}{$\lambda=0$} & \multicolumn{2}{c|}{$\lambda=2$} & \multicolumn{2}{c|}{$\lambda=6$}\\
\cline{2-3} \cline{3-4}\cline{4-7}
 $\tau$ & error  & order &  error & order & error & order\\
\hline
    1/10  & 6.6386e-005 &        &9.6009e-006  &         &8.5068e-007 & \\
    1/20  & 9.2847e-007 & 6.1599 &1.4297e-007  & 6.0694  &1.9943e-008 & 5.4147  \\
    1/40  & 1.5767e-008 & 5.8799 &2.1338e-009  & 6.0661  &3.0437e-010 & 6.0339\\
    1/80  & 2.3505e-010 & 6.0678 &3.5138e-011  & 5.9242  &3.8203e-012 & 6.3159 \\
    1/160 & 3.8498e-012 & 5.9320 &5.3434e-013  & 6.0391  &6.7433e-014 & 5.8241\\ \hline
  \end{tabular}
\label{tabexm3}
\end{table}
In our numerical algorithm, the values of $u(t)$ at points $\{t_{j}\}_{j=0}^{N_I}$ are calculated
by the improved Adam's methods \cite{Deng:07}. The numerical results are reported in Tables \ref{tabexm1}-\ref{tabexm3}.
From Tables \ref{tabexm1}-\ref{tabexm3}, we can see that the convergence orders are in good agreement with
the theory presented in Theorem  \ref{thgldfccr}.

\begin{example}\label{example3}
In this example, we examine the following
initial value problem
\begin{equation}\label{rfrodes3}
_{0}^{C}D_{t}^{\alpha,\lambda}u(t)=-\mu ~u(t),~\mu>0.
\end{equation}
The initial values are given as $e^{\lambda t}u(t)|_{t=0}=1$ and $\big[\frac{d}{dt}(e^{\lambda t}u(t))\big]\big|_{t=0}=0$ for $\alpha\in(1,2)$, and as $e^{\lambda t}u(t)|_{t=0}=1$ for $\alpha\in(0,1)$.
Using the Laplace transform presented in Appendix B, we have
\begin{equation}\label{formula3a}
(s+\lambda)^{\alpha}\widetilde{u}(s)-(s+\lambda)^{\alpha-1}=-\mu~\widetilde{u}(s).
\end{equation}
Then
\begin{equation}\label{formula3b}
\widetilde{u}(s)=\frac{(s+\lambda)^{\alpha-1}}{(s+\lambda)^{\alpha}+\mu}.
\end{equation}
Employing the Laplace transform involving the derivative of
the Mittag-Leffler function \cite{Podlubny:99}
\begin{equation}\label{M-L}
\mathcal {L}\big\{t^{\beta-1}
\big(E_{\alpha,\beta}(-at^{\alpha})\big)\big\}=\frac{s^{\alpha-\beta}}{s^{\alpha}+a},
\quad \textrm{Re}(s)>|a|^{1/\alpha},
\end{equation}
 we can check that the exact solution of this initial value problem is
$$u(t)=e^{-\lambda t}E_{\alpha,1}(-\mu ~t^{\alpha}).$$
Here the generalized Mittag-Leffler function $E_{\alpha,\beta}(\cdot)$ is given by \cite{Podlubny:99}
\begin{equation}\label{DM-L}
E_{\alpha,\beta}(z)
=\sum_{k=0}^{\infty}\frac{z^k}{\Gamma(\alpha k+\beta)},\quad \text{Re}(\alpha)>0.
\end{equation}
\end{example}
In this example, the solution $u(t)$ does not have a bounded first (second) derivative at the initial time $t=0$ as $0<\alpha<1$ ($1<\alpha<2$). To improve the convergence order, we employ the technique given in our previous work \cite{Zhao:13}.
We separately solve the equation in subintervals $[0,T_0]$ and $[T_0,T]$ of the interval $[0,T]$. More specifically, we modify the formula \eqref{numer1} as
\begin{eqnarray}\label{modifiedforb}
u(t)&=& e^{-\lambda t}\sum_{k=0}^{n-1}a_k(t)+\frac{e^{-\lambda t}}{\Gamma(\alpha)}\int^{T_0}_0(t-s)^{\alpha-1}g\big(s,u(s)\big)ds
\nonumber\\
           &&+\frac{e^{-\lambda t}}{\Gamma(\alpha)}\int^{t}_{T_0}(t-s)^{\alpha-1}g\big(s,u(s)\big)ds
\nonumber\\
          &=& e^{-\lambda t}\sum_{k=0}^{n-1}a_k(t)
            +\frac{e^{-\lambda t}}{\Gamma(\alpha)}\int^{T_0}_0(t-s)^{\alpha-1}g\big(s,u(s)\big)ds
\nonumber\\
           &&+\frac{e^{-\lambda t}}{\Gamma(\alpha)}\bigg(\frac{t-T_0}{2}\bigg)^\alpha
           \int^{1}_{-1}(1-z)^{\alpha-1}\tilde{g}\big(z,\tilde{u}(z)\big)dz.
\nonumber
\end{eqnarray}
Here, we suppose that the smoothness of $g$ is weaker on the subinterval  $[0,T_0]$ and sufficiently smooth on $[T_0,T]$. For the integral on the subinterval $[0,T_0]$,
the Gauss-Lobatto quadrature with the weight function $w(s)=1$ is used; and for the one on $[T_0,T]$, we compute it as (\ref{numer2}), i.e.,
\begin{eqnarray}\label{modifiedfora}
x(t)
          &\approx&\sum_{k=0}^{n-1}a_k(t)
          +\frac{1}{\Gamma(\alpha)}\sum_{j=0}^{\tilde{N}}\tilde{\omega}_j(t-s_j)^{\alpha-1}g\big(s_j,u(s_j)\big)
\nonumber\\
           &&+\frac{1}{\Gamma(\alpha)}\bigg(\frac{t-T_0}{2}\bigg)^\alpha\sum_{j=0}^{N}\omega_j\tilde{g}\big(z_j,\tilde {u}(z_j)\big),
\end{eqnarray}
where $\tilde{N},~\{\tilde{\omega}_{j}\}_{j=0}^{\tilde{N}}$ and $\{s_j\}_{j=0}^{\tilde{N}}$
correspond to the number of, the weights of, and the values of the Gauss-Lobatto nodes with the
weight $\omega(s)=1$ in the interval $[0,T_0]$, respectively. The values of
$\big\{g\big(s_j,v(s_j)\big)\big\}_{j=0}^{\tilde{N}}$ can be computed as in the starting
procedure. Since $g$ and $x$ are continuous in the interval $[0,T_0]$, by the theory of Gauss
quadrature \cite{Quarteroni:00} and the analysis above, we can see that if $\tilde{N}$ is a big
number then the accuracy of the total error can still be remained.
 The numerical results are reported in Tables \ref{tabexm4} and \ref{tabexm5}. And it can be seen that the desired
 numerical accuracy is obtained.
\begin{table}[h]
\centering
\caption{Maximum errors and convergence orders of Example \ref{example3} solved by the scheme (\ref{modifiedfora}) with $T=1.1, N=20, \tilde{N}=40, N_I=2, T_0=0.1,\mu=1$, and $\lambda=5$.}
\vspace{0.2cm}
  \begin{tabular}{|c|c|c|c|c|c|c|c|}
\hline
&\multicolumn{2}{c|}{$\alpha=0.2$} & \multicolumn{2}{c|}{$\alpha=0.9$} & \multicolumn{2}{c|}{$\alpha=1.8$}\\
\cline{2-3} \cline{3-4}\cline{4-7}
 $\tau$ & error  & order &  error & order & error & order\\
\hline
    1/20  & 5.4805e-004 &        &1.9043e-005  &         &2.1461e-006 &    \\
    1/40  & 1.8749e-004 & 1.5475 &4.3478e-006  & 2.1309  &5.6685e-007 & 1.9207\\
    1/80  & 5.0838e-005 & 1.8828 &1.0851e-006  & 2.0025  &1.5416e-007 & 1.8786 \\
    1/160 & 1.3492e-005 & 1.9138 &3.1549e-007  & 1.7821  &4.0386e-008 & 1.9324\\ \hline
  \end{tabular}
\label{tabexm4}
\end{table}
\begin{table}[h]
\centering
\caption{Maximum errors and convergence orders of Example \ref{example3} solved by the scheme (\ref{modifiedfora}) with $T=1.1, N=20, \tilde{N}=40, N_I=2, T_0=0.1,\mu=1$, and $\lambda=10$.}
\vspace{0.2cm}
  \begin{tabular}{|c|c|c|c|c|c|c|c|}
\hline
&\multicolumn{2}{c|}{$\alpha=0.2$} & \multicolumn{2}{c|}{$\alpha=0.9$} & \multicolumn{2}{c|}{$\alpha=1.8$}\\
\cline{2-3} \cline{3-4}\cline{4-7}
 $\tau$ & error  & order &  error & order & error & order\\
\hline
    1/20  & 2.0162e-004 &        &7.0054e-006  &         &4.0825e-007 &    \\
    1/40  & 8.8563e-005 & 1.1868 &1.6897e-006  & 2.0517  &1.0286e-007 & 1.9887\\
    1/80  & 2.7211e-005 & 1.7025 &4.1757e-007  & 2.0167  &2.7730e-008 & 1.8912\\
    1/160 & 7.4508e-006 & 1.8688 &1.1169e-007  & 1.9025  &7.3208e-009 & 1.9214\\ \hline
  \end{tabular}
\label{tabexm5}
\end{table}
\section{Concluding remarks}
Currently, it is widely recognized that fractional calculus is a powerful tool in describing anomalous diffusion. Because of the bounded physical space and the finite lifetime of living particles, in the CTRW model, sometimes it is necessary to truncate (temper) the measures of jump length and waiting time with divergent second and first moments, respectively. Exponential tempering offers technique advantages of remaining the infinitely divisible L\'evy process and that the transition densities can be computed at any scale. Then the Fokker-Planck equation of the corresponding CTRW model has the time tempered fractional derivative (tempered waiting time) and/or the space tempered fractional derivative (tempered jump length). This paper focus on discussing the properties of the time tempered fractional derivatives as well as the well-posedness and numerical algorithm for the time tempered evolution equation, i.e., the tempered fractional ODEs. The proposed so-called Jacobi-predictor-corrector algorithm shows its powerfulness/advantages in solving the tempered fractional ODEs, including the one of easily getting any desired convergence orders by simply changing the parameter of the number of the interpolating points and the other one of linearly increasing computational cost with time $t$ rather than quadratically increasing more often happened for numerically solving fractional evolution equation.

\section*{Acknowledgements}
C. Li would like to thanks Prof. Shan Zhao  for his very helpful comments and suggestions.
This research was partially supported by the
National Natural Science Foundation of China under Grant  Nos. 11271173 and 11426174,
 the Starting Research Fund from the Xi'an university of
Technology under Grant Nos. 108-211206 and 2014CX022, and the Scientific
Research Program Funded by Shaanxi Provincial Education Department under Grant No. 2013JK0581.
\begin{appendix}
\renewcommand{\theequation}{A.\arabic{equation}}
\section*{Appendix A: Properties of tempered fractional integral and derivatives}
For the  Riemann-Liouville tempered fractional integral, we have the following properties.
\begin{property} \label {lemmacs}
If $u(t)\in AC^n[a,b]$, $\sigma >0$, and $\lambda\in \mathbb{R}$, then  for all $t\geq a$,
\begin{equation*}
\begin{split}
{_{a}I}_t^{\sigma,\lambda}u(t)=&\sum_{k=1}^{n}\frac{e^{-\lambda t}(t-a)^{\sigma+k-1}  }{\Gamma(\sigma+k)}
\bigg[\frac{d^{k-1} (e^{\lambda t}u(t))}{dt^{k-1}}\bigg|_{t=a}\bigg]
\\&+ \frac{e^{-\lambda t}}{\Gamma(\sigma+n)}\int_{0}^t  (t-s)^{\sigma+n-1} {\frac{d^{n}}{ds^{n}}\big(e^{\lambda s}u(s)\big)}ds.
\end{split}
\end{equation*}
\end{property}
\begin{proof}
If $u(t)$ has continuous derivative for $t\geq a$, then using integration by parts to \eqref{TRLI}, there exists
\begin{equation*}
\frac{1}{\Gamma(\sigma)}\int_{a}^t(t-s)^{\sigma-1}e^{\lambda s}u(s)ds=\frac{(t-a)^{\sigma}u(a)  }{\Gamma(\sigma+1)}
              + \frac{1}{\Gamma(\sigma+1)}\int_{a}^t (t-s)^{\sigma} {\frac{d}{ds}e^{\lambda s}u(s)}ds;
\end{equation*}
and if the function has $n$ continuous derivatives,then integrating by parts, we get
\begin{equation}\label{intecaputo}
\begin{split}
\frac{1}{\Gamma(\sigma)}\int_{a}^t(t-s)^{\sigma-1}e^{\lambda s}u(s)ds=&\sum_{k=1}^{n}\frac{(t-a)^{\sigma+k-1}  }{\Gamma(\sigma+k)}
\bigg[\frac{d^{k-1} (e^{\lambda t}u(t))}{dt^{k-1}}\bigg|_{t=a}\bigg]
\\&+ \frac{1}{\Gamma(\sigma+n)}\int_{0}^t  (t-s)^{\sigma+n-1} {\frac{d^{n}}{ds^{n}}\big(e^{\lambda s}u(s)\big)}ds.
\end{split}
\end{equation}
Multiplying both sides of \eqref{intecaputo} by function $e^{-\lambda t}$, the desired result is obtained.
\end{proof}

In addition, integrating by parts leads to
\begin{equation*}
{_{a}I}_t^{\sigma,\lambda}u(t)= e^{-\sigma t}u(a) +\sigma \int_{a}^t   e^{-\sigma(t-s)}{u(s)}ds + \int_{a}^t   e^{-\sigma(x-s)}d{u(s)},
\end{equation*}
which implies that if $u(t)$ has  continuous derivatives on finite domain $[a,b]$, and $\sigma >0$, then  for all $t\geq a$,
\begin{equation*}
\lim _{\sigma \rightarrow 0} {_{a}I}_t^{\sigma,\lambda}u(t)=u(t).
\end{equation*}
\begin{lemma}\label{lemcnim}
For a function $u(t)\in L([a,b])$, $\sigma>0,\lambda\in \mathbb{R}$, we have
\begin{equation}\label{TCDDb}
      \big\|\frac{1}{\Gamma(\sigma)}\int_{a}^te^{-\lambda( t-s)}(t-s)^{\sigma-1}u(s)ds\big\|_{L([a,b])}\leq  M\|u\|_{L([a,b])},
    \end{equation}
    where $M=\frac{(b-a)^{\sigma}}{\Gamma(\sigma+1)}$.
\end{lemma}
\begin{proof} By simple calculation, we have
\begin{equation*}
\begin{split}
    \big\|\frac{1}{\Gamma(\sigma)}\int_{a}^te^{-\lambda( t-s)}(t-s)^{\sigma-1}u(s)ds\big\|_{L([a,b])}&\leq
     \frac{1}{\Gamma(\sigma)}\int_{a}^{b}\int_{a}^te^{-\lambda( t-s)}(t-s)^{\sigma-1}|u(s)|dsdt      \\
     &\leq
     \frac{1}{\Gamma(\sigma)}\int_{a}^{b}\int_{s}^{b}e^{-\lambda(t-s)}(t-s)^{\sigma-1}dt|u(s)|ds
     \\
     &\leq
     \frac{(b-a)^{\sigma}}{\Gamma(\sigma+1)}\|u\|_{L([a,b])}.
\end{split}
\end{equation*}
\end{proof}

\begin{proposition} \label {lemma2aas}(Semigroup properties)
Let $u(t)\in L([a,b])$ and $\sigma_1,\sigma_2 >0, \lambda\in\mathbb{R}$. Then  for all $t \geq a$,
$$_{a}I_t^{\sigma_1,\lambda}[_{a}I_t^{\sigma_2,\lambda} u(t)]={_{a}I_t^{\sigma_{1}+\sigma_2,\lambda}}u(t)={_{a}I_t^{\sigma_2,\lambda}}[_{a}I_t^{\sigma_1,\lambda} u(t)].$$

%
\end{proposition}
\begin{proof}
Recalling the semigroup property of Riemann-Liouville fractional integral \cite{Podlubny:99}
$I_t^{\sigma_1}I_t^{\sigma_2}u(t)=I_t^{\sigma_{1}+\sigma_{2}}u(t)$, we get the following semigroup property
of tempered fractional integral
\begin{equation*}
\begin{split}
     _{a}I_t^{\sigma_1,\lambda}[_{a}I_t^{\sigma_2,\lambda} u(t)]&=
     e^{-\lambda t}{{_{a}I_t^{\sigma_1}}}\big[ e^{\lambda t}(_{a}I_t^{\sigma_2,\lambda} u(t))\big]      \\
     &=e^{-\lambda t}{{_{a}}I_t^{\sigma_1}}\big[ e^{\lambda t}\big(e^{-\lambda t}{{_{a}}I_t^{\sigma_2}}( e^{\lambda t}u(t))\big)\big]    \\
     &=e^{-\lambda t}{{_{a}}I_t^{\sigma_1}}\big[{{_{a}}I_t^{\sigma_2}}( e^{\lambda t}u(t))\big]   \\
     &=e^{-\lambda t}{{_{a}}I_t^{\sigma_{1}+\sigma_2}}\big( e^{\lambda t}u(t)\big)\\
     &={_{a}I_t^{\sigma_{1}+\sigma_{2},\lambda}} u(t).
\end{split}
\end{equation*}

\end{proof}
Similar to the fractional calculus, the tempered fractional calculus is a linear operation.
\begin{proposition} \label {lemma2ab} (Linearity properties)
Let $u(t)\in L([a,b])$. Then for all $\delta,\mu,\lambda\in \mathbb{R}$: 

(1) for $\sigma >0$, holds
\begin{equation*}
     _{a}I_t^{\sigma,\lambda}[\delta~ u(t)+\mu~ u(t)]=\delta~ {_{a}I_t^{\sigma,\lambda}}[u(t)]+\mu~{_{a}I_t^{\sigma,\lambda}}[u(t)];
\end{equation*}

(2) for $\alpha\in(n-1,n)$, holds
\begin{equation*}
_{a}D_t^{\alpha,\lambda}[\delta~ u(t)+\mu~ u(t)]=\delta~{_{a}D_t^{\alpha,\lambda}}[u(t)]+
\mu~{_{a}D_t^{\alpha,\lambda}}[u(t)].
\end{equation*}

(3) for $\alpha\in(n-1,n)$, holds
\begin{equation*}
_{a}^{C}D_t^{\alpha,\lambda}[\delta~ u(t)+\mu~ u(t)]=\delta~{_{a}^{C}D_t^{\alpha,\lambda}}[u(t)]+
\mu~{_{a}^{C}D_t^{\alpha,\lambda}}[u(t)].
\end{equation*}
\end{proposition}
\begin{proof}
The linearity of fractional integral and derivatives follows directly from the corresponding definitions.
We  omit the details here.
\end{proof}
\begin{proposition} \label {rem4}
Let $u(t)\in C^{n}[a,b]$ and $\alpha \in(n-1,n)$. Then  for all $t\in [a,b]$, there exists
$${_{a}^{C}D}_t^{\alpha,\lambda}[u(t)]|_{t=a}=0.$$
\begin{proof} After simple argument, we have
\begin{equation}\label{TCDD}
      \big|{_{a}^{C}D}_t^{\alpha,\lambda}u(t)\big|\leq\bigg|\frac{e^{-\lambda t}}{\Gamma(n-\alpha)}\int_{a}^t\frac{1}{(t-s)^{\alpha-n+1}}\frac{d^n e^{\lambda s}u(s)}{d
          s^n}ds\bigg|\leq\frac{Me^{-\lambda t}(t-a)^{n-\alpha}}{\Gamma(n-\alpha+1)},
    \end{equation}
    where $M=\max_{t\in[a,b]}\big|\frac{d^n e^{\lambda t}u(t)}{d
          t^n}\big|.$
From the above analysis, we get the desired result.
\end{proof}
\end{proposition}

\renewcommand{\theequation}{B.\arabic{equation}}
\section*{Appendix B: Laplace transforms of the tempered fractional calculus}\label{sectlttfc}
In this subsection, we discuss the Laplace transforms of the tempered fractional calculus.
Define the Laplace transform of a function $u(t)$ and its inverse as \cite{Schiff91}

\begin{equation}\label{deflaplce}
\mathcal {L}\{u(t) ;s\}=\widetilde{u}(s)
=\int^{+\infty}_{0}e^{-st}u(t)dt,
\end{equation}
and
\begin{equation}\label{indeflaplce}
\mathcal {L}^{-1}\{\widetilde{u}(s) ;t\}=u(t)
=\frac{1}{2\pi i}\int^{c_0+i\infty}_{c_0-i\infty}e^{st}\widetilde{u}(s)ds,\quad c_0=\text{Re}(s)>0,\quad i^2=-1.
\end{equation}
We start with the Laplace transform of the Riemann-Liouville tempered fractional integral of order $\sigma$.
\begin{proposition}\label{lem0af}
 The Laplace transform of the  Riemann-Liouville tempered fractional integral is given by
  \begin{equation}\label{ffb}
     \mathcal{L}({_{0}}I_t^{\sigma,\lambda}u(t))=(\lambda+s)^{-\sigma}\widetilde{u}(s).
  \end{equation}
\end{proposition}
\begin{proof}
First, we rewrite the Riemann-Liouville tempered fractional integral as the form of convolution
\begin{equation*}
_{0}I_t^{\sigma,\lambda}u(t)=\frac{1}{\Gamma(\sigma)}\int_{0}^te^{-\lambda( t-s)}(t-s)^{\sigma-1}u(s)ds=\frac{e^{-\lambda t}t^{\sigma-1}}{\Gamma(\sigma)}\ast u(t).
\end{equation*}
In view of the Laplace transform of the convolution \cite{Schiff91}
\begin{equation*}
\mathcal {L}\{u(t)\ast v(t) ;s\}=\widetilde{u}(s)\widetilde{v}(s),
\end{equation*}
we have
\begin{equation*}
\mathcal {L}\{ {_{0}I}_t^{\sigma,\lambda}u(t);s\}=\mathcal {L}\left\{\frac{e^{-\lambda t}t^{\sigma-1}}{\Gamma(\sigma)};s\right\}
\mathcal {L}\{u(t);s\}.
\end{equation*}
Recalling the Laplace transform
$
\mathcal {L}\{ e^{-\lambda t}t^{\sigma-1};s\}=\Gamma(\sigma)(\lambda+s)^{-\sigma},
$
we have the Laplace transform of Riemann-Liouville tempered fractional integral
\begin{equation}\label{intlaplace}
\mathcal {L}\{ {_{0}I}_t^{\sigma,\lambda}u(t);s\}=(\lambda+s)^{-\sigma}\widetilde{u}(s).
\end{equation}
\end{proof}
Next, we turn to consider the Laplace transform of tempered fractional derivative.
\begin{proposition}\label{lem0b}
The Laplace transform of the  Riemann-Liouville tempered fractional derivative is given by
  \begin{equation}\label{Lpa}
\mathcal {L}\{ {_{0}D}_t^{\alpha,\lambda}u(t);s\}=(s+\lambda)^{\alpha}\widetilde{u}(s)
-\sum_{k=0}^{n-1}(s+\lambda)^{k}\bigg[{_{0}D}_t^{\alpha-k-1}(e^{\lambda t}u(t))\big|_{t=0}\bigg].
\end{equation}

The Laplace transform of the  Caputo tempered fractional derivative is given by
\begin{equation}\label{capuLpa}
\mathcal {L}\{ {_{0}^{C}D}_t^{\alpha,\lambda}u(t);s\}=(s+\lambda)^{\alpha}\widetilde{u}(s)-\sum_{k=0}^{n-1}(s+\lambda)^{\alpha-k-1}
\bigg[\frac{d^{k}}{dt^{k}}(e^{\lambda t}u(t))\big|_{t=0}\bigg].
\end{equation}
\end{proposition}
\begin{proof}
Using the properties of Riemann-Liouville tempered fractional calculus, we may rewrite the  Riemann-Liouville tempered fractional derivative
as
\begin{equation*}
_{0}D_t^{\alpha,\lambda}u(t)
          =e^{-\lambda t}\frac{\mathrm{d}^n}{\mathrm{d}t^n}\big(v(t)\big),\quad n-1<\alpha<n,
\end{equation*}
where $v(t)$ denotes
\begin{equation*}
 v(t)=\frac{1}{\Gamma(n-\alpha)}\int_{0}^t\frac{e^{\lambda s}u(s)}{(t-s)^{\alpha-n+1}}ds
 ={_{0}I}_t^{n-\alpha}(e^{\lambda t}u(t)).
 \end{equation*}
 Furthermore, using the formula of Laplace transform of  an integer order derivative
 \begin{equation*}
\mathcal {L}\left\{ \frac{d^n}{d
t^n}v(t);s\right\}=s^{n}\widetilde{v}(s)-\sum_{k=0}^{n-1}s^{n-k-1}\bigg[\frac{\mathrm{d}^k}{\mathrm{d}t^k}v(t)\bigg]\bigg|_{t=0}
=s^{n}\widetilde{v}(s)-\sum_{k=0}^{n-1}s^{k}\bigg[\frac{\mathrm{d}^{n-k-1}}{\mathrm{d}t^{n-k-1}}v(t)\bigg]\bigg|_{t=0}.
 \end{equation*}
 Combing the first translation Theorem \cite{Schiff91}
 \begin{equation}\label{translaplace}
\mathcal {L}\{ e^{-\lambda t}u(t);s\}=\widetilde{u}(\lambda+s),\quad \lambda\in R, \quad Re(s)>\lambda,
\end{equation}
we have
\begin{equation}\label{laplacecca}
\mathcal {L}\left\{ e^{-\lambda t}\frac{d^n}{d
t^n}v(t);s\right\}
=(s+\lambda)^{n}\widetilde{v}(s+\lambda)-\sum_{k=0}^{n-1}(s+\lambda)^{k}
\bigg[\frac{\mathrm{d}^{n-k-1}}{\mathrm{d}t^{n-k-1}}v(t)\bigg]\bigg|_{t=0},
 \end{equation}
which implies
\begin{equation}\label{laplaces}
\mathcal {L}\{ {_{0}D}_t^{\alpha,\lambda}u(t);s\}=(s+\lambda)^{n}\widetilde{v}(s+\lambda)-\sum_{k=0}^{n-1}(s+\lambda)^{k}
\bigg[\frac{\mathrm{d}^{n-k-1}}{\mathrm{d}t^{n-k-1}}v(t)\bigg]\bigg|_{t=0}.
 \end{equation}
By applying the convolution theorem once again, we get
 \begin{equation*}
 \begin{split}
\mathcal {L}\{ v(t);s\}=&\mathcal {L}\{{_{0}I}_t^{n-\alpha}(e^{\lambda t}u(t))\}\\
=&\mathcal {L}\left\{ \frac{t^{n-\alpha-1}}{\Gamma(n-\alpha)};s\right\}\mathcal {L}\{ (e^{\lambda t}u(t));s\}\\
=&s^{-(n-\alpha)}\widetilde{u}(s-\lambda).
\end{split}
 \end{equation*}
Combining above formulae, we have
\begin{equation}\label{lapace0a}
\widetilde{v}(s+\lambda)=(s+\lambda)^{-(n-\alpha)}\widetilde{u}(s).
 \end{equation}
In addition, from the definition of the Riemann-Liouville fractional  integral it follows that
 \begin{equation}\label{lapace0}
\frac{\mathrm{d}^{n-k-1}}{\mathrm{d}t^{n-k-1}}v(t)=\frac{d^{n-k-1}}{d
t^{n-k-1}}{_{0}I}_t^{n-\alpha}(e^{\lambda t}u(t))={_{0}D}_t^{-(k-\alpha+1)}\big(e^{\lambda x}u(t)\big).
 \end{equation}
 Substituting \eqref{lapace0a} and \eqref{lapace0} into \eqref{laplaces}, we have
 \begin{equation}\label{rmlapacea0}
\mathcal {L}\{ {_{0}D}_t^{\alpha,\lambda}u(t);s\}=(s+\lambda)^{\alpha}\widetilde{u}(s)
-\sum_{k=0}^{n-1}(s+\lambda)^{k}{_{0}D}_t^{-(k-\alpha+1)}\big(e^{\lambda x}u(t)\big)\big|_{t=0}.
\end{equation}

To establish the Laplace transform for the Caputo tempered fractional  derivative,
let us write it in the form
\begin{equation*}
\begin{split}
{_{0}^{C}D}_t^{\alpha,\lambda}u(t)=&e^{-\lambda t}~{_{0}^{C}D}_t^{\alpha}\big(e^{\lambda t}u(t)\big)
=e^{-\lambda t}~{_{0}I}_t^{n-\alpha}\big(v(t)\big),\\
v(t)
=&\frac{d^n}{d
t^n}\big(w(t)\big),\\
w(t)
=&e^{\lambda t}u(t).
\end{split}
\end{equation*}
Applying the Riemann-Liouville fractional integral  and the first translation Theorem \eqref{translaplace}, we then have
that
\begin{equation}\label{caputolapacea0o}
\mathcal {L}\{ {_{0}^{C}D}_t^{\alpha,\lambda}u(t);s\}=(s+\lambda)^{-(n-\alpha)}\widetilde{v}(s+\lambda),
\end{equation}
where
\begin{equation}\label{in0tlaplaceo}
\begin{split}
\widetilde{v}(s+\lambda)=&(s+\lambda)^{n}\widetilde{w}(s+\lambda)-\sum_{k=0}^{n-1}(s+\lambda)^{\alpha-k-1}
\bigg[\frac{d^k}{dt^k}w(t)\bigg]\bigg|_{t=0}.
\end{split}
\end{equation}
Noticing that
$
\mathcal {L}\{ w(t);s\}=\widetilde{u}(s-\lambda),
$
and inserting \eqref{in0tlaplaceo} into \eqref{caputolapacea0o}, we arrive at the Laplace transform of the Caputo tempered fractional  derivative \eqref{capuLpa}.
\end{proof}

It is easy to see that if taking $\lambda=0$ in \eqref{Lpa}, we have the Laplace transform 
of Riemann-Liouville fractional derivative \cite{Podlubny:99}
\begin{equation}\label{1.01a}
   \mathcal {L}\{ {_{0}D}_t^{\alpha}u(t);s\}=s^{\alpha}\widetilde{u}(s)
-\sum_{k=0}^{n-1}s^{k}\bigg[{_{0}D}_t^{-(k-\alpha+1)}(u(t))\big|_{t=0}\bigg].
\end{equation}
And if taking $\lambda=0$ in \eqref{capuLpa}, we can get the Laplace transform 
of Caputo fractional derivative \cite{Podlubny:99}
\begin{equation}\label{1.01b}
    \mathcal {L}\{_{0}^{C}D_{t}^{\alpha}u(t) ;s\}=
    s^{\alpha}\widetilde{u}(s)-\sum_{k=0}^{n-1}s^{\alpha-k-1}\bigg[\frac{d^{k} u(t)}{d t^{k}}\bigg|_{t=0}\bigg],\quad n-1<\alpha< n.
\end{equation}
From the Laplace transform of tempered fractional derivatives, we  observe that 
different initial value conditions are needed for fractional differential equations with different fractional derivatives.
From \eqref{capuLpa}, it can be noted that the Laplace transform for the  Caputo tempered fractional derivative
involves the values of the function $u(t)$ and its derivatives at the lower terminal $t=0$, which are easily specified in physical. So the Caputo type fractional derivatives are more popularly used in time direction \cite{Podlubny:99}.
The Laplace transform for the variants of the  Riemann-Liouville tempered fractional derivatives are given as 
\begin{equation}\label{Rrla}
\mathcal{L}(_{-\infty}\textbf{D}_t^{\alpha,\lambda}u(t))=
\begin{cases}
\displaystyle
(\lambda+s)^{\alpha}\widetilde{u}(s)-\big[{_{0}I}_t^{1-\alpha}(e^{\lambda t}u(t))|_{t=0}\big]-\lambda^{\alpha}\widetilde{u}(s),~ \text{$0<\alpha<1$},\\
\\
\displaystyle
(\lambda+s)^{\alpha}\widetilde{u}(s)-\alpha \lambda^{\alpha-1} (s\widetilde{u}(s)-u(0))-\lambda^{\alpha}\widetilde{u}(s)
-\big[{_{0}D}_t^{\alpha-1}(e^{\lambda t}u(t))|_{t=0}\big]\\
-(\lambda+s)\big[{_{0}I}_t^{2-\alpha}(e^{\lambda x}u(t))\big|_{t=0}\big],~\text{$1<\alpha<2$}.
\end{cases}
\end{equation}
\end{appendix}

\end{document}